\documentclass[12pt]{extarticle}
\usepackage[utf8]{inputenc}
\usepackage[left=3cm,right=2.5cm]{geometry}
\usepackage{graphicx}
\usepackage{amsmath}
\usepackage{amsfonts,mathrsfs}
\usepackage{amssymb}
\usepackage{soul}
\usepackage{color}
\usepackage{tikz}
\usepackage{stmaryrd}
\usepackage{ytableau}
\usepackage{float}
\usepackage[english]{babel}
\usepackage{ytableau}
\newtheorem{theorem}{Theorem}[section]

\newtheorem{lemma}{Lemma}[section]
\newtheorem{proposition}{Proposition}[section]
\newenvironment{proof}[1][Proof]{\noindent\textbf{#1.} }{\ \rule{0.5em}{0.5em}}

\newcounter{ourcount}
\newcounter{myenumi}
\newcommand{\Sn}[2]{\mathfrak{S}_{#1}^{#2}}
\newcommand{\Snfree}[2]{\widehat{\mathfrak{S}}_{#1,#2}}

\newcommand{\w}[2]{\big[#1,#2\big]}

\newcommand{\wpp}[1]{\llbracket #1\rrbracket}
\newcommand{\wB}[2]{\mathsf{B}_{#1,#2}\left(\delta\right)}
\newcommand{\wBfree}[2]{\widehat{\mathsf{B}}_{#1,#2}}
\newcommand{\qwB}[2]{\mathsf{qB}_{#1,#2}\left(\delta,q\right)}

\newcommand{\qn}[1]{\left(#1\right)_q}

\newcommand{\Binomial}[2]{\left(\begin{array}{c}
    #1\\
    #2 
\end{array}\right)}

\newcommand{\plambda}{\boldsymbol{\lambda}}

\newcommand{\sfu}{\mathsf{u}}
\newcommand{\sfv}{\mathsf{v}}
\newcommand{\sfw}{\mathsf{w}}

\newcommand{\bec}{\vartriangleleft}
\newcommand{\bbec}{\blacktriangleleft}

\makeatletter
\newcommand*{\rom}[1]{\expandafter\@slowromancap\romannumeral #1@}
\makeatother

\begin{document}
$\ $
\vspace{1.cm}
\begin{center}
{\Large \textbf{Normal form for the walled Brauer algebra:\\
construction and applications
}}

{\small
\vspace{.45cm} 
{\large \textbf{D.V. Bulgakova$^{1,2}$, Y.O. Goncharov$^{5,6}$, O.V. Ogievetsky$^{1,2,3,4}$ }}

\vskip .35cm \texttt{dvbulgakova@gmail.com $\qquad$ yegor.goncharov@gmail.com}

\vskip .5cm $^{1}$Aix Marseille Universit\'{e}, Universit\'{e} de
Toulon, CNRS, \\ CPT UMR 7332, 13288, Marseille, France

\vskip .35cm $^{2}$ I.E.Tamm Dept of Theoretical Physics, Lebedev Physical Institute,\\ 
Leninsky prospect 53, 119991, Moscow, Russia

\vskip .35cm $^{3}$ Bogoliubov Laboratory of Theoretical Physics,\\
JINR, 141980 Dubna, Moscow region, Russia

\vskip .35cm $^{4}$ Kazan Federal University, Kremlevskaya 17, Kazan 420008, Russia

\vskip .35cm $^{5}$ Service de Physique de l’Univers, Champs et Gravitation,
Universit\'e de Mons - UMONS,\\ 20 Place du Parc, B-7000 Mons, Belgique

\vskip .35cm $^{6}$ Institut Denis Poisson,
Universit\'e de Tours, Universit\'e d’Orl\'eans, CNRS,\\
Parc de Grandmont, 37200 Tours, France
}

\end{center}

\begin{abstract}\noindent
We construct a normal form for the walled Brauer algebra, together with the reduction algorithm. We apply normal form to calculate the numbers of monomials in generators with minimal length. We further utilize normal form to give explicit expressions for a generating set and annihilator ideal of a particular cyclic vector in a cell module.
\end{abstract}

\section{Introduction}
The first studies of the walled Brauer algebra ${\sf B}_{r,s}(\delta)$, see \cite{T,K,BCHLLS}, were motivated by interest in a version of the Schur-Weyl duality for the group $GL_{\delta}(\mathbb{C})$. If $\delta\in\mathbb{N}$, 
the duality relates mutually commuting actions of ${\sf B}_{r,s}(\delta)$ and $GL_{\delta}(\mathbb{C})$ on the mixed tensor product $V^{\otimes r}\otimes (V^*)^{\otimes s}$ of the natural representation and its dual for $GL_{\delta}(\mathbb{C})$.

\vskip .1cm
The walled Brauer algebra ${\sf B}_{r,s}(\delta)$ is an associative unital $(r+s)!$-dimensional algebra\footnote{According to \cite{NV}: ``The history of the definition of this algebra is as follows. Turaev \cite{T} was the first to define it by a presentation; he also pointed out to the second author that it is $(r+s)!$-dimensional and resembles the group algebra of the symmetric group. The walled Brauer algebra was independently defined in \cite{K}''.} defined for all $\delta\in\mathbb{C}$. It is generated by elements $s_i$, $i=1,\ldots ,r+s-1$, with the following defining relations (see, {\it e.g.}, \cite{BS,JK})
\begin{align}
s_i^2&=1,\quad i\neq r ,\label{rel0}\\
s_r^2&=\delta s_r ,\label{rel1}\\
s_is_{i+1}s_i&=s_{i+1}s_is_{i+1},\quad i,i+1\neq r,\label{rel5}\\
s_i s_j&=s_j s_i \; \text{if}\;  |i-j|>1,\label{rel6}\\
s_r s_{r\pm 1} s_r&= s_r ,\label{rel2}\\
s_r s_{r+1} s_{r-1} s_r s_{r-1}&=s_r s_{r+1} s_{r-1} s_r s_{r+1} ,\label{rel3}\\
\ s_{r-1} s_r s_{r+1} s_{r-1} s_r&=s_{r+1} s_r s_{r+1} s_{r-1} s_r .\label{rel4}
\end{align}
Note that the elements $s_i$ with $1\leqslant i<r$ (respectively, $r<i<r+s$) generate the symmetric group algebra $\mathbb{C}\mathfrak{S}_r$ (respectively, $\mathbb{C}\Sn{s}{}$). The algebra $\wB{r}{s}$ contains $\wB{r}{0}\cong \mathbb{C}\Sn{r}{}$ and $\wB{0}{s}\cong \mathbb{C}\Sn{s}{}$ as commuting subalgebras, together generating $\mathbb{C}\left[\Sn{r}{}\times\Sn{s}{}\right]$.

\vskip .1cm
In the present paper we construct a normal form $\mathfrak{B}_{r,s}$ for $\wB{r}{s}$ -- a set of basis monomials (words) in generators $s_i$. 
To construct the set $\mathfrak{B}_{r,s}$ we introduce an `ordered' modification of the so-called Bergman's diamond lemma \cite{B}, namely, we present a set of rules which, being applied in a {\it certain} order, allows to reduce any monomial in generators to an element from $\mathfrak{B}_{r,s}$. 

\vskip .1cm
The normal form $\mathfrak{B}_{r,s}$ possesses a number of useful properties. The normal form has the following factorized form
\begin{equation}\label{eq:normf_intro}
    \mathfrak{B}_{r,s} = \bigcup_{f = 0}^{\min(r,s)} \mathfrak{B}_{r,s}^{(f)},\;\text{where}\quad \mathfrak{B}_{r,s}^{(f)}=\Sn{r}{L}\mathfrak{D}_{r,s}^{(f)}\Sn{s}{R}.
\end{equation}
Here $\Sn{r}{L}$ and $\Sn{s}{R}$ are normal forms for the symmetric groups, while the set $\mathfrak{D}_{r,s}^{(f)}$ contains $f$ times the noninvertible generator $s_r$, see details in Section \ref{Normf_Alg}. 

\vskip .1cm
The normal form \eqref{eq:normf_intro} agrees with the natural embeddings $\mathsf{A}_{p,q}\subset\wB{r}{s}$, 
where $\mathsf{A}_{p,q}$ is generated by the subset $s_{p+1},\ldots,s_{r+s -1-q}$ of generators. Each $\mathsf{A}_{p,q}$ itself is a walled Brauer algebra and our normal form for the $\mathsf{A}_{p,q}$ is obtained by leaving in $\mathfrak{B}_{r,s}$ only monomials built from generators of a corresponding subset. 
Note that the particular chain $\wB{1}{0}\subset \wB{2}{0} \subset\ldots\subset \wB{r}{0}\subset \wB{r}{1}\subset\ldots \subset \wB{r}{s}$ of the above embeddings is important in the representation theory of the walled Brauer algebra
\cite{CVDM,JK,BO}. 
 
\vskip .1cm 
Factorized structure of \eqref{eq:normf_intro} will be utilized to find the generating function for the numbers
$\nu_{\ell}$ of independent words with a given minimal length $\ell$ (Lemma \ref{wordswB}):
\begin{equation}
    F_{r,s}\left(q\right) = \sum_{\ell} \nu_{\ell}\,q^{\ell} = \qn{r + s}!,
\end{equation}
where $\qn{n} := 1 + q +\ldots + q^{n-1}$ and $\qn{n}! := \qn{1}\cdot \ldots\cdot\qn{n}$. Remarkably, the result coincides with the analogous generating function for the symmetric group
 $\Sn{r+s}{}$.

\vskip .1cm
Representation theory of the walled Brauer algebra is well understood. In \cite{CVDM} it was shown that cell modules arising from a certain cellular algebra structure on ${\sf B}_{r,s}(\delta)$ are labeled by pairs of Young diagrams $\plambda_f=(\lambda_f^L,\lambda_f^R)$ such that $r-|\lambda^L_f| = s-|\lambda^R_f| = f$. The basis of the module $C_{r,s}(\plambda_f)$ was given in terms of the so-called partial one-row diagrams, see Section \ref{Modules_1}. The semisimplicity criterion for $\wB{r}{s}$ was given in \cite{CVDM}. In this paper we always assume that $\delta$ is generic, that is, the walled Brauer algebra is semisimple.

\vskip .1cm
In the present paper we describe the cell modules in terms of left ideals in $\wB{r}{s}$. Namely we calculate the annihilator ideal of a particular vector $v_f\in C_{r,s}(\plambda_f)$ in a module. We construct a basis of the annihilator ideal of a vector $v_f$ using the normal form \eqref{eq:normf_intro} and give a constructive proof of the following Theorem. 

\begin{theorem}\label{decomp}
Fix a $\wB{r}{s}$-module $C_{r,s}(\plambda_f)$. Let $A_f$ and $X_f$ be the sets given in lemmas \ref{genset} and \ref{annihilator}. Then the union $A_f\cup X_f$ is a basis of 
the algebra $\wB{r}{s}$.\end{theorem}

It appears that the normal form \eqref{eq:normf_intro} is also appropriate for the $q$-deformed algebra  
$\qwB{r}{s}$. In \cite{KM} the basis of $\qwB{r}{s}$ analogous to \eqref{eq:normf_intro} was introduced for a specific value of $\delta$ and generalized to all values in \cite{H}. 

\vskip .3cm
The paper is organized as follows. In Section \ref{Normf_Alg} we recall the definition of the walled Brauer algebra and construct the normal form $\mathfrak{B}_{r,s}$. In Section \ref{Modules} we recall the construction of the basis of the module in terms of the partial one-row diagrams. There we also construct the annihilator ideal for a given module, thus giving a constructive proof for the Theorem \ref{decomp}. Some explicit calculations and details are postponed to the Appendix.

\section{Normal form for the walled Brauer algebra}\label{Normf_Alg}

Monomials in generators $s_i$, $1\leqslant i<r+s$, whose lengths cannot be reduced by any composition of relations \eqref{rel0}-\eqref{rel4} 
will be referred to as minimal words. It may happen that an element of the algebra $\wB{r}{s}$ can be represented by several monomials of the same length 
in view of relations \eqref{rel5}, \eqref{rel6}, \eqref{rel3}, \eqref{rel4} which do not affect monomial lengths.

\vskip .1cm
In this Section we shall consider bases of the algebra $\wB{r}{s}$ consisting of elements which can be represented by minimal words.
By a normal form for the algebra $\wB{r}{s}$ we mean a basis of $\wB{r}{s}$ and a unique choice of a word representing each basis element.

\vskip .1cm
To construct a normal form we make use of Bergman's diamond lemma \cite{B}. Let $\wBfree{r}{s}$ denote the monoid freely generated by 
elements $\hat{s}_i$, $1\leqslant i<r+s$.  Let also $\wBfree{r}{s}\langle\hat\delta\rangle$ denote the monoid freely generated by elements $\hat{s}_i$, $1\leqslant i<r+s$, and a central element $\hat\delta$. For a subset $\mathfrak{E}$ of $\wBfree{r}{s}$ we denote by 
$\mathfrak{E}\langle\hat\delta\rangle$ the subset of $\wBfree{r}{s}\langle\hat\delta\rangle$ consisting of words $\hat\delta^j \mathsf{e}\,$ for $\mathsf{e}\in\mathfrak{E}$ and
$j=0,1,2,\dots$

\vskip .1cm
We propose a reduction system $\mathfrak{R}$, a set of words $\mathfrak{B}_{r,s}\subset \wBfree{r}{s}$ and an 
algorithm $\varphi_{\mathfrak{R}}:\wBfree{r}{s}\langle\hat\delta\rangle\to 
\mathfrak{B}_{r,s}\langle\hat\delta\rangle$ transforming any given monomial to a particular reduced form. We show that the image of $\mathfrak{B}_{r,s}$ under the natural map 
\begin{equation}\label{namap}
\mathbb{C}\wBfree{r}{s}\langle\hat\delta\rangle\to \wB{r}{s}\ ,\ \hat{s}_i\mapsto s_i,\hat\delta\mapsto\delta\ ,\end{equation} 
forms a basis of the algebra $\wB{r}{s}$. 

\vskip .1cm
Reduction system $\mathfrak{R}$ is constituted by ordered pairs $\rho=(\mathsf{w}_{\rho},\mathsf{w}^{\prime}_{\rho})$ of monomials $\mathsf{w}_{\rho}\in\wBfree{r}{s}$,
$\mathsf{w}_{\rho}\neq 1$, and    
$\mathsf{w}^{\prime}_{\rho}\in \mathfrak{B}_{r,s}\langle\hat\delta\rangle$; such a pair is written as $\mathsf{w}_{\rho}\to \mathsf{w}^{\prime}_{\rho}$ and understood as the 
substitution instruction, or {\it reduction}: the instruction, applied to a word $\mathsf{e}$, chooses   
a subword, equal to the {\it lhs} and replaces it by the {\it rhs}.  
A monomial is called irreducible if no reduction can be applied to it.

\vskip .1cm
Reductions can be subject to ambiguities meaning that more than one instruction from $\mathfrak{R}$ can be applicable to a given monomial. All ambiguities are analyzed in terms of the following two elementary ones \cite{B}. 
If $\sfv_1\sfv_2 = \sfw_{\rho}$ and $\sfv_2\sfv_3 = \sfw_{\tau}$, where $\sfv_1,\sfv_2,\sfv_3\neq 1$, for some $\rho,\tau\in\mathfrak{R}$ one faces an alternative of transforming $\sfv_1\sfv_2\sfv_3$ either into $\sfw^\prime_{\rho}\sfv_3$ or into $\sfv_1\sfw^\prime_{\tau}$. This is called an overlap ambiguity of $\mathfrak{R}$. If $\sfv_2=\sfw_{\rho}$ and $\sfv_1\sfv_2\sfv_3 = \sfw_{\tau}$, where $\sfv_1\neq 1$ or $\sfv_3\neq 1$, one can transform $\sfv_1\sfv_2\sfv_3$ either into $\sfv_1\sfw^\prime_{\rho}\sfv_3$ or into $\sfw^\prime_{\tau}$. This is referred to as inclusion ambiguity. An ambiguity of $\mathfrak{R}$ is said to be resolvable when there exist reductions $\varphi_1,\varphi_2$ such that $\varphi_1(\sfw^\prime_{\rho}\sfv_3) = \varphi_2(\sfv_1\sfw^\prime_{\tau})$ in case of an overlap and $\varphi_1(\sfv_1\sfw^\prime_{\rho}\sfv_2) = \varphi_2(\sfw^\prime_{\tau})$ in case of an inclusion. 

\vskip .1cm
For generators $\hat{s}_i$, $1\leqslant i<r+s$, denote the word $\hat{s}_p \hat{s}_{p-1}\dots \hat{s}_q\in\wBfree{r}{s}$ ($1\leqslant q\leqslant p < r+s$) by $[p,q]$, and set $[q-1,q]=1$ by definition. 
\begin{proposition}\label{rules}
Let $\mathfrak{R}$ be the following reduction system
\begin{align}
         \hat{s}_i^2 &\rightarrow 1,\quad i\neq r,\label{red1}\\
          \hat{s}_j\hat{s}_i &\rightarrow \hat{s}_i\hat{s}_j,\,\,\,j-i>1,\label{red3}\\
        \hat{s}_{i+1}\hat{s}_{i}\dots \hat{s}_{i-j}\hat{s}_{i+1}&\rightarrow \hat{s}_i\hat{s}_{i+1}\hat{s}_i\dots \hat{s}_{i-j},\,\,\,i<r-1,\,0\leqslant j<i,\label{red4}\\
        \hat{s}_i\hat{s}_{i+j}\dots \hat{s}_{i+1}\hat{s}_i&\rightarrow \hat{s}_{i+j}\dots \hat{s}_{i+1}\hat{s}_i\hat{s}_{i+1},\,\,\,i>r,\,1\leqslant j<r+s-i,\label{red5}\\
        \hat{s}_r^2&\rightarrow\hat\delta \hat{s}_r,\label{red2}\\
        \hat{s}_r\hat{s}_{r-1}\dots \hat{s}_{r-i}\hat{s}_r&\rightarrow \hat{s}_{r-2}\dots \hat{s}_{r-i}\hat{s}_r,\,\,\, 1\leqslant i<r,\label{red6}\\
        \hat{s}_r\hat{s}_{r+i}\dots \hat{s}_{r+1}\hat{s}_r&\rightarrow \hat{s}_r\hat{s}_{r+i}\dots \hat{s}_{r+2},\,\,\,1\leqslant i<s,\label{red7}\\
        [r+j,r-i][r+j,r]&\rightarrow \hat{s}_{r-1}[r+j-1,r-i][r+j,r],\,\,\,1\leqslant i<r,\,1\leqslant j<s,\label{red8}\\
        [r,r-i][r+j,r-i]&\rightarrow [r,r-i][r+j,r-i+1]\hat{s}_{r+1},\,\,\,1\leqslant i<r,\,1\leqslant j<s\label{red9}.
\end{align}
Then

\noindent {\rm (i)} All ambiguities of $\mathfrak{R}$ are resolvable.

\noindent {\rm (ii)} The factor-algebra of the monoid algebra $\mathbb{C}\wBfree{r}{s}\langle\hat\delta\rangle$ by the ideal generated by the elements $\mathsf{w}_{\rho}-\mathsf{w}^{\prime}_{\rho}$, for $\rho$ ranging through the set of instructions, and $\hat\delta-\delta$, is isomorphic to the walled Brauer algebra $\wB{r}{s}$.

\end{proposition}
\begin{proof}
(i) We start by turning defining relations for $\wB{r}{s}$ \eqref{rel0}-\eqref{rel4} into the following set $\mathfrak{R}_0$ of instructions
\begin{align*}
         \hat{s}_i^2 &\rightarrow 1,\quad i\neq r,\\
          \hat{s}_j\hat{s}_i &\rightarrow \hat{s}_i\hat{s}_j,\,\,\,j-i>1,\\
       \hat{s}_{i+1}\hat{s}_{i}\hat{s}_{i+1}&\rightarrow \hat{s}_i\hat{s}_{i+1}\hat{s}_i,\,\,\,i<r-1,\\
        \hat{s}_i\hat{s}_{i+1}\hat{s}_i&\rightarrow \hat{s}_{i+1}\hat{s}_i\hat{s}_{i+1},\,\,\,i>r,\\
        \hat{s}_r^2&\rightarrow\hat\delta \hat{s}_r,\\
        \hat{s}_r\hat{s}_{r\pm 1}\hat{s}_r&\rightarrow \hat{s}_r,\\
        \hat{s}_{r+1}\hat{s}_r\hat{s}_{r-1}\hat{s}_{r+1}\hat{s}_r&\rightarrow \hat{s}_{r-1}\hat{s}_r\hat{s}_{r-1}\hat{s}_{r+1}\hat{s}_r,\\
        \hat{s}_r\hat{s}_{r-1}\hat{s}_{r+1}\hat{s}_r\hat{s}_{r-1}&\rightarrow \hat{s}_r\hat{s}_{r-1}\hat{s}_{r+1}\hat{s}_r\hat{s}_{r+1},
\end{align*}
which is a subset $\mathfrak{R}_0\subset\mathfrak{R}$. It is straightforward to check that the reduction system $\mathfrak{R}_0$ is free from inclusion ambiguities while overlap ambiguities are not resolvable unless one recursively extends $\mathfrak{R}_0$ to the reduction system $\mathfrak{R}$. The latter is subject only to overlap ambiguities as well. Resolvability of these ambiguities can be verified by a successive check considering first  all ambiguities of \eqref{red1} with \eqref{red3}-\eqref{red9} then all ambiguities of \eqref{red3} with \eqref{red4}-\eqref{red9} etc.

The assertion (ii) follows since the instructions from $\mathfrak{R}$ are consequences of the instructions from $\mathfrak{R}_0$. 
\end{proof}

\vskip .1cm
To guarantee that a reduction system $\mathfrak{R}$ leads to a set of irreducible words  in a finite number of steps, 
Theorem 1.2 \cite{B} assumes the existence of a partial order $<$ on the set of free monomials  
such that: i) $\sfw_1<\sfw_2$ implies $\sfu\sfw_1\sfv< \sfu\sfw_2\sfv$ for all $\sfu,\sfv$, ii) $<$ is compatible with $\mathfrak{R}$ in a sense that $\sfw^{\prime}_{\rho}<\sfw_{\rho}$ for each instruction $\sfw_{\rho}\to\sfw^{\prime}_{\rho}$, iii) any chain $\sfv_1> \sfv_2> \ldots$ terminates. For the system $\mathfrak{R}$ such order does not exist. Indeed, assume that it does.  Applying the rule \eqref{red9}  to the monomial $\hat{s}_r \hat{s}_{r-1}\hat{s}_{r+1}\hat{s}_r \hat{s}_{r-1}\hat{s}_r\hat{s}_{r-1}\hat{s}_{r+1}\hat{s}_r$ we arrive at $\hat{s}_r \hat{s}_{r-1}\hat{s}_{r+1}\hat{s}_r \hat{s}_{r+1}\hat{s}_r\hat{s}_{r-1}\hat{s}_{r+1}\hat{s}_r$, so we must have
\begin{equation}
    \hat{s}_r \hat{s}_{r-1}\hat{s}_{r+1}\hat{s}_r \hat{s}_{r+1}\hat{s}_r\hat{s}_{r-1}\hat{s}_{r+1}\hat{s}_r<\hat{s}_r \hat{s}_{r-1}\hat{s}_{r+1}\hat{s}_r \hat{s}_{r-1}\hat{s}_r\hat{s}_{r-1}\hat{s}_{r+1}\hat{s}_r.  
\end{equation}
Then, by applying \eqref{red8} to the result, one gets the opposite relation
\begin{equation}
    \hat{s}_r \hat{s}_{r-1}\hat{s}_{r+1}\hat{s}_r \hat{s}_{r-1}\hat{s}_r\hat{s}_{r-1}\hat{s}_{r+1}\hat{s}_r<\hat{s}_r \hat{s}_{r-1}\hat{s}_{r+1}\hat{s}_r \hat{s}_{r+1}\hat{s}_r\hat{s}_{r-1}\hat{s}_{r+1}\hat{s}_r,    
\end{equation}
which is a contradiction. 
This example shows that some sequences of instructions from $\mathfrak{R}$ do not terminate so the reduction system $\mathfrak{R}$, directly understood, does not lead to a normal form. We shall not investigate the question about the existence of another reduction system compatible with a certain order. 
Instead, we will present a trick allowing to construct a well-defined algorithm $\varphi_{\mathfrak{R}}$ which uses precisely the reduction system $\mathfrak{R}$. 
Namely we will specify the order of applying the rules from $\mathfrak{R}$. 

\vskip .1cm
For that purpose we split $\mathfrak{R} = \mathfrak{R}^\prime\cup \mathfrak{R}^{\prime\prime}$, with $\mathfrak{R}^\prime$ constituted by instructions \eqref{red1}-\eqref{red7} and $\mathfrak{R}^{\prime\prime}$ -- by \eqref{red8}, \eqref{red9}. For the set $\mathfrak{R}^\prime$ a partial order $\vartriangleleft$, satisfying conditions i)-iii), on $\wBfree{r}{s}$ 
does exist; it is described in 
Appendix A. Therefore, the reduction system $\mathfrak{R}^{\prime}$ leads to a set of irreducible words $\mathfrak{B}_{r,s}^{\star}$ and gives a well-defined algorithm 
$\varphi_{\mathfrak{R}^\prime}:\wBfree{r}{s}\langle\hat\delta\rangle
\to \mathfrak{B}_{r,s}^{\star}\langle\hat\delta\rangle$.
To describe the words from $\mathfrak{B}_{r,s}^\star$, we first note that instructions \eqref{red1}-\eqref{red5} move generators of $\Sn{r}{}$ (respectively, $\Sn{s}{}$) to the left (respectively, to the right) and arrange them into a certain normal form 
(it is of no importance at the moment and will be specified later). With this in hand, it is a straightforward exercise to check that $\mathfrak{B}_{r,s}^\star$ is constituted by monomials of the form
\begin{equation}\label{eq:Bstar}
    \sfw^{i_1,\dots,i_f}_{j_1,\dots,j_f}=\sfw_L\w{r+i_1}{r-j_1}\w{r+i_2}{r-j_2}\ldots \w{r+i_f}{r-j_f}\sfw_R   
\end{equation}
where $\sfw_{L}\in\Sn{r}{}$ and $\sfw_{R}\in\Sn{s}{}$ 
are in a normal form, $r>i_1\geqslant 0$, $s>j_f\geqslant 0$, $s>i_2,\ldots,i_f\geqslant 1$, $r>j_1,\ldots,j_{f-1}\geqslant 1$ and $0\leqslant f\leqslant\min(r,s)$. 

\vskip .1cm
Clearly, the instructions from $\mathfrak{R}^{\prime\prime}$ do not preserve the set $\mathfrak{B}_{r,s}^\star$.
We specify the algorithm $\varphi_{\mathfrak{R}\prime\prime}$ of applying the reductions from $\mathfrak{R}^{\prime\prime}$ to the monomials $\sfw\in\mathfrak{B}_{r,s}^\star$.
Assume, for a monomial $\sfw$ of the form \eqref{eq:Bstar} with $f\geqslant 2$, that the set $\{ j_1,\dots,j_f\}$ is not strictly decreasing. 
Then there exists the maximal value $k = 1\ldots f-1$ such that $j_k\leqslant j_{k+1}$. The word $\sfw$ contains a subword $\w{r}{r-j_k}\w{r+i_{k+1}}{r-j_k}$ and we apply the instruction \eqref{red9} to it, obtaining $\w{r}{r-j_k}\w{r+i_{k+1}}{r-j_k+1}\hat{s}_{r+1}$. Reductions of this kind (call them $\psi$) break the structure \eqref{eq:Bstar}, $\psi(\sfw)\notin\mathfrak{B}_{r,s}^\star$. It is straightforward to check that $\sfw$ and $\varphi_{\mathfrak{R}^\prime}\circ \psi(\sfw)\in\mathfrak{B}_{r,s}^\star$ have the same $f$, while,  for the word $\varphi_{\mathfrak{R}^\prime}\circ \psi(\sfw)$, the maximal $k^\prime = 1\ldots f-1$ such that $j_{k^\prime}<j_{k^\prime+1}$ (if exists) is less than $k$. 
Iterating this procedure, we arrive at the word which has the form \eqref{eq:Bstar} with $j_1> j_2>\ldots>j_f\geqslant 0$ and thus is irreducible with respect to the union of
$\mathfrak{R}^\prime$ and \eqref{red9}. 
As soon as the ordering in $j$'s is achieved, we start to apply, in a similar way, the instruction \eqref{red8} to arrive at the ordering
$0\leqslant i_1< i_2<\ldots<i_f$ of $i$'s (now we look for the minimal $l=1\ldots f-1$ such that $i_{l-1}\geq i_l$).

\vskip .1cm
Our final algorithm $\varphi_{\mathfrak{R}}$ is the composition of the algorithm $\varphi_{\mathfrak{R}\prime\prime}$ and the algorithm $\varphi_{\mathfrak{R}\prime}$,  
$\varphi_{\mathfrak{R}}=\varphi_{\mathfrak{R}\prime\prime}\circ\varphi_{\mathfrak{R}\prime}$.

\vskip .1cm
We have established the following Proposition.
\begin{proposition} The set $\mathfrak{B}_{r,s}$ of irreducible words with respect to the algorithm $\varphi_{\mathfrak{R}}$ consists of the monomials $ \sfw^{i_1,\dots,i_f}_{j_1,\dots,j_f}$ 
of the form (\ref{eq:Bstar}) with $0\leqslant i_1< i_2<\ldots<i_f<r$ and $s>j_1> j_2>\ldots>j_f\geqslant 0$.\end{proposition}

\begin{lemma}\label{warmup}
The set $\mathfrak{B}_{r,s}$  
contains $(r+s)!$ elements.
\end{lemma}
\begin{proof}
The set of monomials with a given $f$ is in bijection with the product of the set of subsets of cardinality $f$ in a set of cardinality $r$ by the set of subsets of cardinality $f$ in a set of cardinality $s$, so
\begin{equation}\label{eq:cardB2}
    \# \mathfrak{B}_{r,s} = r!s!\,\sum_{f = 0}^{\min(r,s)}\left(\begin{array}{c}
        s  \\
         f   
     \end{array}\right)
     \left(\begin{array}{c}
        r  \\
        f   
     \end{array}\right)=(r+s)!\ .
\end{equation}
\end{proof}

\begin{lemma}\label{normf}
The image of the set $\mathfrak{B}_{r,s}$ under the map (\ref{namap}) forms a basis in $\wB{r}{s}$.
\end{lemma}
\begin{proof}
By construction, the images of the words from $\mathfrak{B}_{r,s}$ are linearly independent. The assertion follows, since the cardinality of $\mathfrak{B}_{r,s}$ coincides, by Lemma \ref{normf}, with the dimension of 
$\wB{r}{s}$.
\end{proof}

\vskip .1cm
Multiplication in $\wB{r}{s}$ is expressed in terms of the basis monomials $\sfv_1,\sfv_2\in\mathfrak{B}_{r,s}$ as $\varphi_{\mathfrak{R}}(\sfv_1\sfv_2)\in \mathfrak{B}_{r,s}$. Left multiplication by generators $s_i$ $\varphi_{\mathfrak{R}}(s_i\sfv)$, $s_i, \sfv\in\mathfrak{B}_{r,s}$, is presented in Appendix B.

\vskip .1cm
Note that the Lemma \ref{normf} holds for any choice of normal forms for $\sfw_{L}\in\Sn{r}{}$ and $\sfw_{R}\in\Sn{s}{}$. The ones we consider in this work are obtained via the reduction system $\mathfrak{R}$. We denote by $\Sn{r}{L}$ the set of words in normal form $[1,1-i_1]\dots[r-1,r-1-i_{r-1}]$ with $-1\leqslant i_1<1,\ldots,-1\leqslant i_{r-1}< r-1$ for the symmetric group $\Sn{r}{}$, and by $\Sn{s}{R}$ the set of words in normal form $[r+1+j_{s-1},r+1]\dots[r+s-1+j_{1},r+s-1]$ with $-1\leqslant j_1<1,\ldots,-1\leqslant j_{s-1}<s-1$ for $\Sn{s}{}$. 

\vskip 0.2cm
We denote by $\mathfrak{D}_{r,s}^{(f)}$ be the set of words 
\begin{equation}\label{eq:Crs}
\w{r+i_1}{r-j_1}\w{r+i_2}{r-j_2}\ldots \w{r+i_f}{r-j_f}
\end{equation}
with $0\leqslant i_1< i_2<\ldots<i_f<r$ and $s>j_1> j_2>\ldots>j_f\geqslant 0$. We set $\mathfrak{D}_{r,s}^{(0)}=\{1\}$. 

\vskip .1cm
In this notation the set $\mathfrak{B}_{r,s}$ decomposes as
\begin{equation}\label{eq:normf}
    \mathfrak{B}_{r,s} = \bigcup_{f = 0}^{\min(r,s)} \mathfrak{B}_{r,s}^{(f)},\;\text{where}\quad \mathfrak{B}_{r,s}^{(f)}=\Sn{r}{L}\mathfrak{D}_{r,s}^{(f)}\Sn{s}{R}.
\end{equation}

Let $\nu_{\ell}$ be the number of words of length $\ell$ in $\mathfrak{B}_{r,s}$ and $F_{r,s}\left(q\right) = \sum_{\ell} \nu_{\ell}\,q^{\ell}$ the corresponding generating function.
\begin{lemma}\label{wordswB} We have
\begin{equation}
    F_{r,s}\left(q\right) = \qn{r + s}!\ ,
\end{equation}
where $(m)_q:=1+q+q^2+\dots+q^{m-1}$ denotes the quantum number $m$, and $\qn{m}! := \qn{1}\cdot \ldots\cdot\qn{m}$.
\end{lemma}
\begin{proof}
The generating function for the numbers of words of given length $F_{r,s}\left(q\right)$ for $\mathfrak{B}_{r,s}$ has the factorized form
\begin{equation}\label{eq:generatingF}
    F_{r,s}\left(q\right) = F_{r}\left(q\right)\tilde{F}_{r,s}\left(q\right) F_{s}\left(q\right),
\end{equation}
where $F_{r}\left(q\right)$ (respectively, $F_{s}\left(q\right)$) are generating functions for $\Sn{r}{L}$ (respectively, $\Sn{s}{R}$), while $\tilde{F}_{r,s}\left(q\right)$ is a generating function for $\bigcup_{f }\mathfrak{D}_{r,s}^{(f)}$. The length of the word (\ref{eq:Crs}) is $f+\sum_a i_a+\sum_b j_b$ so the generating function $\tilde{F}_{r,s}\left(q\right)$ is
easily found using, e.g;, Theorem 6.1 in \cite{KC},
$$\tilde{F}_{r,s}\left(q\right)=\sum_f q^{f^2}\Binomial{r}{f}_q\Binomial{s}{f}_q=\Binomial{r+s}{r}_q$$
by the $q$-Vandermonde identity. Here $\Binomial{a+b}{b}_q:=\frac{(a+b)_q!}{(a)_q!(b)_q!}$ is the $q$-binomial coefficient. The rest follows.
\end{proof}

\vskip .1cm
Abusing notation we will denote by the symbol $\mathfrak{B}_{r,s}$ the image of the set $\mathfrak{B}_{r,s}$ in the algebra $\wB{r}{s}$. As well, we will denote the word 
$s_p s_{p-1}\dots s_q$ by the symbol $[p,q]$.

\section{Modules over $\wB{r}{s}$}\label{Modules}
\subsection{Diagrammatical description of $\wB{r}{s}$}
Aside from the definition of $\wB{r}{s}$ as a factor-algebra of $\wBfree{r}{s}$, there is also a convenient graphical presentation for a basis of $\wB{r}{s}$ in terms of the so-called walled diagrams, which are defined as follows. Let $p^u_{r,s} =  p^u_r\cup p^u_s$ and $p^d_{r,s} = p^d_r\cup p^d_s$ be two sets, each consisting of $r+s$ nodes aligned horizontally on the plane. The nodes in the set $p^d_{r,s}$ are placed under the nodes in the set $p^u_{r,s}$ and a vertical wall separates the first $r$ nodes $p^u_r$ ($p^d_r$) in the upper (lower) row from the last $s$ nodes $p^u_s$ ($p^d_s$). A walled diagram $d$ is a bijection between the set $p^u_{r,s}\cup p^d_{r,s}$ and visualised by placing the edges between the corresponding points in the following way:

\begin{itemize}
    \item[1.] edges connecting nodes between $p^u_{r,s}$ and $p^d_{r,s}$ do not cross the wall (we call them propagating lines),
    \item[2.] edges connecting nodes between $p^u_{r,s}$ and $p^u_{r,s}$ and  between $p^d_{r,s}$ and $p^d_{r,s}$ cross the wall (we call them arcs).
\end{itemize}
Let $\delta$ be a complex parameter. As a vector space, the walled Brauer algebra ${\sf B}_{r,s}(\delta)$ is identified with the $\mathbb{C}$-linear span of the walled diagrams. The product of two basis elements $\mathfrak{d}_2\mathfrak{d}_1$ is obtained by placing $\mathfrak{d}_1$ above $\mathfrak{d}_2$ and identifying the nodes of the top row of $\mathfrak{d}_2$ with the corresponding nodes in the bottom row of $\mathfrak{d}_1$. Let $\ell$ be the number of closed loops so obtained. 
The product $\mathfrak{d}_1\mathfrak{d}_2$ is given by $\delta^{\ell}$ times the resulting diagram with loops omitted. 

\vskip .2cm
The following walled diagrams represent the generators $s_i$ (the vertical dotted line represents the wall): 

\begin{equation}\label{eq:generators}
\begin{array}{l}
\begin{tikzpicture}
\draw (0,0) --(0,1.2);
\draw (1.5,0) --(2,1.2);
\draw (1.5,1.2) --(2,0);
\draw (3.5,0) --(3.5,1.2);
\draw [dotted] (4,-.2) --(4,1.4);
\draw (4.5,0) --(4.5,1.2);
\draw (6,0) --(6,1.2);
\node at (0,0) {\tiny\textbullet};
\node at (0,1.2) {\tiny\textbullet};
\node at (1.5,0) {\tiny\textbullet};
\node at (2,1.2) {\tiny\textbullet};
\node at (1.5,1.2) {\tiny\textbullet};
\node at (2,0) {\tiny\textbullet};
\node at (3.5,0) {\tiny\textbullet};
\node at (3.5,1.2) {\tiny\textbullet};
\node at (4.5,0) {\tiny\textbullet};
\node at (4.5,1.2) {\tiny\textbullet};
\node at (6,0) {\tiny\textbullet};
\node at (6,1.2) {\tiny\textbullet};
\node at (.8,1.2) {$\dots$};
\node at (.8,0) {$\dots$};
\node at (2.8,1.2) {$\dots$};
\node at (2.8,0) {$\dots$};
\node at (5.3,1.2) {$\dots$};
\node at (5.3,0) {$\dots$};
\node at (0,-.4) {$1$};
\node at (1.3,-.4) {$i$};
\node at (2.2,-.42) {$i+1$};
\node at (3.5,-.43) {$r$};
\node at (4.8,-.4) {$r+1$};
\node at (6.3,-.4) {$r+s$};
\node at (6.5,.5) {$,$};
\node at (9,.5) {$\ \ \ 1\leqslant i <r+s,$};
\node at (9.2,0) {$i\neq r$};
\node at (-1.2,.5) {$s_i:=$};
\end{tikzpicture}
\\
\begin{tikzpicture}
\draw (0,0) --(0,1.2);
\draw (1.5,0) --(1.5,1.2);
\draw (2.1,0) .. controls (2.2,0.4) and (2.7,0.4) .. (2.8,0);
\draw [dotted] (2.45,-.2) --(2.45,1.4);
\draw (2.1,1.2) .. controls (2.2,0.8) and (2.7,0.8) .. (2.8,1.2);
\draw (3.4,0) --(3.4,1.2);
\draw (4.9,0) --(4.9,1.2);
\node at (0,0) {\tiny\textbullet};
\node at (0,1.2) {\tiny\textbullet};
\node at (1.5,0) {\tiny\textbullet};
\node at (1.5,1.2) {\tiny\textbullet};
\node at (2.1,0) {\tiny\textbullet};
\node at (2.8,0) {\tiny\textbullet};
\node at (2.1,1.2) {\tiny\textbullet};
\node at (2.8,1.2) {\tiny\textbullet};
\node at (3.4,0) {\tiny\textbullet};
\node at (3.4,1.2) {\tiny\textbullet};
\node at (4.9,0) {\tiny\textbullet};
\node at (4.9,1.2) {\tiny\textbullet};
\node at (.8,1.2) {$\dots$};
\node at (.8,0) {$\dots$};
\node at (4.2,1.2) {$\dots$};
\node at (4.2,0) {$\dots$};
\node at (0,-.4) {$1$};
\node at (2.0,-.42) {$r$};
\node at (2.95,-.40) {$r+1$};
\node at (5.0,-.40) {$r+s$};
\node at (-1.2,.5) {$s_r:=$};
\node at (5.4,0.5) {$.$};
\end{tikzpicture}\\
\end{array}
\end{equation}

\subsection{$\wB{r}{s}$-modules}\label{Modules_1}

Modules over $\wB{r}{s}$, induced from simple modules over $\mathbb{C}\left[\Sn{r}{}\times \Sn{s}{}\right]\subset\wB{r}{s}$, are referred to as cell modules \cite{CVDM}. 

\vskip .1cm
Let $\lambda=(\lambda_1,\lambda_2,\dots)$ be a partition; $\lambda_1,\lambda_2,\dots$ are non-negative integers, $\lambda_1\geqslant\lambda_2\geqslant\dots$. Let $\displaystyle|\lambda|=\sum_{ i\geqslant 1}\lambda_i$. To each partition $\lambda$ we associate its Young diagram -- a left-justified array of rows of boxes containing $\lambda_1$ boxes in the top row, $\lambda_2$ boxes in the second row, etc. A bipartition is a pair of partitions $\plambda=(\lambda^L,\lambda^R)$. We denote by $\Lambda$ the set of all bipartitions. For each integer $0\leqslant f\leqslant \min(r,s)$, we set
\begin{align}\label{eq:Lambda}
    &\!\!\!\! \Lambda_{r,s}(f):=\{ \plambda_f=(\lambda_f^L,\lambda_f^R)\in\Lambda\,:\, \mid r-|\lambda_f^L|=s-|\lambda_f^R|=f\}\ \text{and}\ \Lambda_{r,s}:=\!\!\!\!\!\bigcup_{f=0}^{\min(r,s)}\!\!\Lambda_{r,s}(f).
\end{align}

Simple $\wB{r}{s}$-modules are indexed by elements of the set $\Lambda_{r,s}$ (see \cite{CVDM}). The module indexed by $\plambda_f$ is denoted by $C_{r,s}(\plambda_f)$. Standard tableaux $t^L_f$ (respectively, $t^R_f$) of the shape $\lambda^L_f$ (respectively, $\lambda^R_f$) parameterize basis vectors $\left|t^{L}_f\,\right>$ (respectively, $\left|t^{R}_f\,\right>$) of the Specht module $S(\lambda_f^L)$ (respectively, $S(\lambda_f^R)$) over $\Sn{r-f}{}$ 
(respectively, $\Sn{s-f}{}$).   Choose subsets $l'=\{a'_1,\dots,a'_{f}\}\subset\{1,\dots,r\}$ and
$l=\{a_1,\dots,a_{f}\}\subset\{r+1,\dots,r+s\}$ and an  isomorphism $l^\prime \rightarrow l$ between them. 
There is a basis of the module $C_{r,s}(\plambda_f)$ with the basis vectors  
\begin{equation}\label{eq:vectors}
    \left|l'\rightarrow l,\,t^{L}_f,\,t^{R}_f\,\right>.
\end{equation}
 Vectors \eqref{eq:vectors} admit a graphical presentation in terms of the so-called `partial one-row' diagrams \cite{CVDM}, see Fig. \ref{fig:vect}. 
\begin{figure}[H]
    \begin{equation*}
\ytableausetup{boxsize=1em}
\left|\left(2\to 11, 4\to 8, 6\to 9\right), \begin{ytableau} 1 & 2 \\3 \end{ytableau},\, \begin{ytableau} 1 & 2 \end{ytableau}\,\right\rangle = 
\begin{array}{c}
\begin{tikzpicture}
\draw (0,0) --(0,0.9);
\draw (0.8,0) --(0.4,0.9);
\draw (1.6,0) --(0.8,0.9);
\draw (0.4,0) .. controls (1.2,0.8) and (3.6,0.8) .. (4.4,0);
\draw (2.0,0) .. controls (2.35,0.48) and (3.25,0.48) .. (3.6,0);
\draw (1.2,0) .. controls (1.6,0.5) and (2.8,0.5) .. (3.2,0);
\draw [dotted] (2.4,-.2) --(2.4,1.0);
\draw (2.8,0) --(4.0,0.9);
\draw (4.0,0) --(4.4,0.9);
\node at (0,0) {\tiny\textbullet};
\node at (0,0.9) {\tiny\textbullet};
\node at (0.8,0) {\tiny\textbullet};
\node at (1.6,0) {\tiny\textbullet};
\node at (0.8,0.9) {\tiny\textbullet};
\node at (0.42,0) {\tiny\textbullet};
\node at (0.42,0.9) {\tiny\textbullet};
\node at (2.0,0) {\tiny\textbullet};
\node at (1.2,0) {\tiny\textbullet};
\node at (4.4,0) {\tiny\textbullet};
\node at (4.4,0.9) {\tiny\textbullet};
\node at (3.6,0) {\tiny\textbullet};
\node at (3.2,0) {\tiny\textbullet};
\node at (4.0,0) {\tiny\textbullet};
\node at (4.0,0.9) {\tiny\textbullet};
\node at (2.8,0) {\tiny\textbullet};
\node at (0,-.4) {$1$};
\node at (1,-.4) {$\dots$};
\node at (2.0,-.4) {$6$};
\node at (3.4,-.4) {$\dots$};
\node at (4.4,-.4) {$11$};
\node at (0.4,1.6) {$\begin{ytableau}
1 & 2\\
3 
\end{ytableau} $};
\node at (4.25,1.5) {$\begin{ytableau}
1 & 2
\end{ytableau} $};
\end{tikzpicture}
\end{array}
\end{equation*}
\vskip -0.3cm
    \caption{an example of a vector for ${\sf B}_{6,5}(\delta)$.}
    \label{fig:vect}
\end{figure}
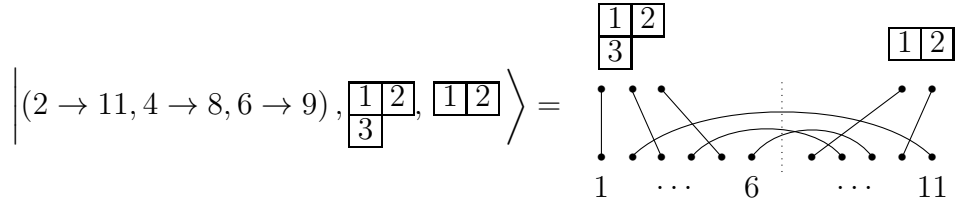

Extending the terminology for the walled diagrams, we call lines starting at tableaux {\it propagating} lines of the partial one-row diagram; other lines will of course be called
arcs.  

\vskip .1cm
We shall define the action of the algebra 
$\wB{r}{s}$ on the vector space $C_{r,s}(\plambda_f)$. To this end, it is sufficient to define the action 
of  a walled diagram $\mathfrak{d}$ from $\wB{r}{s}$ on a partial one-row diagram
$\mathfrak{v}_f$ with $f$ arcs. Place $\mathfrak{d}$ under $\mathfrak{v}_f$ and identify the nodes of $\mathfrak{v}_f$ with the corresponding nodes in the top row of $\mathfrak{d}$. 
This is not necessarily a partial one-row diagram: two propagating lines might start to form an arc. In this case the result of action is zero. Otherwise, 
let $\ell$ be the number of closed loops obtained after the above identification. Omitting the loops we obtain some one-row diagram $\bar{\mathfrak{v}}_f$.
The diagram $\bar{\mathfrak{v}}_f$ may also contain intersections of propagating lines. We numerate the propagating lines of $\mathfrak{v}_f$ by $1,\dots r-f$ on the left of the wall and by $1,\dots s-f$ on the right. Let $\pi_L$ and $\pi_R$ be permutations of $1,\dots r-f$ and $1,\dots s-f$ respectively such that the application of $\pi_L\pi_R$ to the propagating lines' ends of $\mathfrak{v}_f$ gives $\bar{\mathfrak{v}}_f$. The result of the action of $\mathfrak{d}$ on
$\mathfrak{v}_f$ is the combination of 
the partial one-row diagrams obtained from $\bar{\mathfrak{v}}_f$ by forgetting the intersections of propagating lines and writing out the result of the action $\pi_L \left|t^{L}_f\,\right>$ and $\pi_R \left|t^{R}_f\,\right>$ on the vectors of the modules $S(\lambda^L_f)$ and $S(\lambda^R_f)$.

\vskip .1cm
Consider the following vector in the module $C_{r,s}(\plambda_f)$
\begin{equation}\label{eq:rainbow}
   v_f = \big|(r\rightarrow r+1,\,r-1\rightarrow r+2,\ldots, r-f+1\rightarrow r+f),\,\check{t}^{L}_f,\,\check{t}^{R}_f\,\rangle,
\end{equation}
where $\check{t}^{L}_f$ and $\check{t}^{R}_f$ are filled with numbers $1\dots r-f$ and $1\dots s-f$, respectively, in natural order reading down the column from left to right (for an example, see Fig. \ref{rho}).
\begin{figure}[H]
    \begin{equation*}
\left|\left(4\to 9,5\to 8, 6\to 7 \right), \begin{ytableau} 1&3\\ 2 \end{ytableau},\begin{ytableau} 1&2 \end{ytableau} \right\rangle =
\begin{array}{c}
\begin{tikzpicture}
\draw (0,0) --(0,0.9);
\draw (0.4,0) --(0.4,0.9);
\draw (0.8,0) --(0.8,0.9);
\draw (1.2,0) .. controls (1.7,0.7) and (3.1,0.7) .. (3.6,0);
\draw (1.6,0) .. controls (1.9,0.55) and (2.9,0.55) .. (3.2,0);
\draw (2.0,0) .. controls (2.1,0.4) and (2.7,0.4) .. (2.8,0);
\draw [dotted] (2.4,-.2) --(2.4,1.0);
\draw (4.0,0) --(4.0,0.9);
\draw (4.4,0) --(4.4,0.9);
\node at (0,-.4) {$1$};
\node at (1,-.4) {$\dots$};
\node at (2.0,-.4) {$6$};
\node at (3.4,-.4) {$\dots$};
\node at (4.4,-.4) {$11$};
\node at (0,0) {\tiny\textbullet};
\node at (0,0.9) {\tiny\textbullet};
\node at (0.8,0) {\tiny\textbullet};
\node at (0.8,0.9) {\tiny\textbullet};
\node at (1.6,0) {\tiny\textbullet};
\node at (0.4,0.9) {\tiny\textbullet};
\node at (0.4,0) {\tiny\textbullet};
\node at (2.0,0) {\tiny\textbullet};
\node at (1.2,0) {\tiny\textbullet};
\node at (4.4,0) {\tiny\textbullet};
\node at (4.4,0.9) {\tiny\textbullet};
\node at (3.6,0) {\tiny\textbullet};
\node at (3.2,0) {\tiny\textbullet};
\node at (4.0,0) {\tiny\textbullet};
\node at (4.0,0.9) {\tiny\textbullet};
\node at (2.8,0) {\tiny\textbullet};
\node at (0.4,1.6) {$\begin{ytableau}
1 & 3\\
2 
\end{ytableau} $};
\node at (4.25,1.5) {$\begin{ytableau}
1 & 2
\end{ytableau} $};
\end{tikzpicture}
\end{array}
\end{equation*}
\vskip -0.3cm
    \caption{the vector $v_3$ for ${\sf B}_{6,5}(\delta)$}
    \label{rho}
\end{figure}
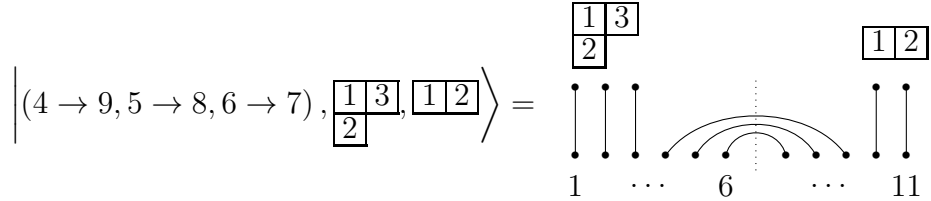

Consider the set $\mathsf{Sh}^L_{r-f,f}\subset\Sn{r}{L}$ (respectively, $\mathsf{Sh}^R_{f,s-f}\subset\Sn{s}{R}$) of words $[1+i_1,1][2+i_2,2]\dots[r-f+i_{r-f},r-f]$ with $-1\leqslant i_1\leqslant i_2\leqslant \dots\leqslant i_{r-f}<f$ (respectively, $[r+1+j_{f},r+1]\dots[r+f+j_{1},r+f]$ with $-1\leqslant j_f\leqslant\dots\leqslant j_1\leqslant s-f$).
The elements of the set $\mathsf{Sh}^L_{r-f,f}$ (respectively, $\mathsf{Sh}^R_{f,s-f}$) represent 
$(r-f,f)$-shuffles (respectively, $(f,s-f)$-shuffles). 

\vskip .1cm
Let $\Sn{f}{L}\subset\Sn{r}{L}$ (respectively, $\Sn{f}{R}\subset\Sn{s}{R}$) be a subset of all monomials in $\Sn{r}{L}$ which include only generators $s_{r-f+1},\ldots,s_{r-1}$ (respectively, $s_{r+1},\ldots,s_{r+f-1}$) for $f>0$. We suppose $\Sn{0}{L} = \left\{1\right\}$ and $\Sn{0}{R} = \left\{1\right\}$. In other words, the elements of $\Sn{f}{L}$
(respectively, $\Sn{f}{R}$) are permutations of the nodes $\{r-f+1,\dots,r\}$ (respectively, $\{r+1,\dots,r+f\}$).

\vskip .1cm
Let $\Theta_f$  with $f\geqslant 0$ denote the following set of permutations from $\Sn{r}{}\times \Sn{s}{}$:
\begin{equation}\label{eq:theta}
    \Theta_f = \mathsf{Sh}^L_{r-f,f}\,\Sn{f}{L}\mathsf{Sh}^R_{f,s-f}.
\end{equation}
It is straightforward to see that the cardinality of $\Theta_f$ is
 $\left(
    \begin{array}{c}
      r \\
      f
    \end{array}
  \right)
  \left(
    \begin{array}{c}
      s \\
      f
    \end{array}
\right)f!$. 

The set $\Theta_f$ 
contains those and only those permutations, from $\Sn{r}{}\times \Sn{s}{}$, of the nodes of the partial one-row diagram corresponding to the vector  $v_f$ which do not permute the propagating lines of the diagram.
Thus $\Theta_f$ produces all possible subsets $l^\prime$ and $l$ of cardinality $f$ and isomorphisms $l^\prime\to l$ as in \eqref{eq:vectors}, {\it i.e.}
\begin{equation}
    \Theta_f v_f= \left\{ \left|l^\prime\to l, \check{t}^L_f,\check{t}^R_f\right>\right\}.
\end{equation}
Let $\Sigma^L_f$ (respectively, $\Sigma^R_f$) 
be the set of all permutations of $\{1,\dots,r-f\}$ (respectively, $\{r+f+1,\dots,r+s\}$)
such that $\sigma_L \check{t}_f^{L}$ and $\sigma_R \check{t}_f^{R}$ reproduce all possible standard tableaux. Let $\Sigma_f$ be the set constituted by permutations 
$\sigma=\sigma_L\sigma_R$ with $\sigma_L\in\Sigma^L_f$ and $\sigma_R\in\Sigma^R_f$. In particular, $\# \Sigma_f = \dim S(\lambda^L_f)\dim S(\lambda^R_f)$.

\vskip .1cm
We introduced the sets $\Sigma^L_f$ and $\Sigma^R_f$ in order to generate vectors \eqref{eq:vectors} with all possible standard tableaux. 
Namely, define the set $X_f$ of permutations from $\Sn{r}{}\times \Sn{s}{}$
\begin{equation}\label{eq:Xf}
    X_f = \Theta_f\Sigma_f. 
\end{equation}
\begin{lemma}\label{genset}
The set of vectors $X_f v_f$ forms a basis of $C_{r,s}(\plambda_f)$. We have
\begin{equation}\label{eq:dim}
    \dim C_{r,s}(\plambda_f)= \# X_f =\left(
    \begin{array}{c}
      r \\
      f
    \end{array}
  \right)
  \left(
    \begin{array}{c}
      s \\
      f
    \end{array}
  \right)f!\;\dim S(\lambda_f^L)\,\dim S(\lambda_f^R).
\end{equation}
\end{lemma}

\subsection{Annihilator ideal} 
We proceed by describing the ideal annihilating the vector $v_f\in C_{r,s}(\plambda_f)$. 
The basis $\mathfrak{B}_{r,s}$ will be convenient for that purpose. We associate to any monomial
\begin{equation}
\mathfrak{x}=\w{r+i_1}{r-j_1}\w{r+i_2}{r-j_2}\ldots \w{r+i_f}{r-j_f}\in\mathfrak{D}_{r,s}^{(f)} 
\end{equation}
the element
\begin{equation}\label{varpimfrx}
\varpi (\mathfrak{x}):=\w{r+i_1}{r-j_1}\w{r+i_2}{r-j_2}\ldots \w{r+i_f}{r+1}.
\end{equation}
We denote by $\bar{\mathfrak{D}}_{r,s}^{(f)}$ the image of the set $\mathfrak{D}_{r,s}^{(f)}$,
$$\bar{\mathfrak{D}}_{r,s}^{(f)}:=\left\{\varpi(\mathfrak{x})\ \vert\ \mathfrak{x}\in\mathfrak{D}_{r,s}^{(f)}\right\}.$$
In words, to construct elements in $\bar{\mathfrak{D}}_{r,s}^{(f)}$ we delete the ends $\w{r}{r-j_{f}}$ of the 
monomials in $\mathfrak{D}_{r,s}^{(f)}$.

\vskip .1cm
Note that $0\leqslant i_1< i_2<\ldots<i_f<r$ and $s>j_1> j_2>\ldots>j_f\geqslant 0$ for an element $\mathfrak{x}\in\mathfrak{D}_{r,s}^{(f)}$ 
so $0\leqslant i_1< i_2<\ldots<i_f<r$ and $s>j_1> j_2>\ldots>j_{f-1}> 0$ for the element (\ref{varpimfrx}). 

\vskip .1cm
We define the product of an element 
$$\mathfrak{y}=\w{r+i_1}{r-j_1}\ldots \w{r+i_f}{r+1}\in\bar{\mathfrak{D}}_{r,s}^{(f)}
$$
and the monomial $\w{r}{r-k}$, $k\geq 0$, to be 
$$\mathfrak{y} \ast \w{r}{r-k}:=\left\{ \begin{array}{lll} \w{r+i_1}{r-j_1}\ldots \w{r+i_f}{r-k}&\text{if}&j_{f-1}>k\ ,
\\ \varnothing&&\text{otherwise}\ .\end{array}\right. $$ 
We introduce the sets
\begin{equation}\label{eq:bprime}
    \bar{\mathfrak{B}}_{r,s}^{(t)}= \Sn{r}{L}\bar{\mathfrak{D}}_{r,s}^{(t)},\;t=1\dots\min(r,s).
\end{equation}
We describe the basis of the annihilator ideal of the vector $v_f$ in three steps. 

\paragraph{Part 1.} Let us introduce the following sets of elements of the algebra 
$\wB{r}{s}$:  
\begin{align}
    &\bigcup\nolimits_{i=0}^{f-1} \left(\w{r}{r-i} \w{r+i}{r+1}^{-1}-\delta\right), \label{eq:loop1}\\
    &\bigcup\nolimits_{i = 0}^{\min(f,s-1)-1}\left(\w{r}{r-i} \w{r+i+1}{r+1}^{-1}-1\right), \label{eq:loop2}\\
    &\bigcup\nolimits_{i = 1}^{\min(f,r-1)}\left(\w{r}{r-i} \w{r+i-1}{r+1}^{-1}-1\right), \label{eq:loop3}\\
    &\w{r}{r-f}\w{r+f}{r+1}^{-1} \label{eq:loop4}.
\end{align}
\vskip 0.2cm
\noindent
The elements \eqref{eq:loop1}-\eqref{eq:loop4} annihilate the vector $v_f$ which is clear from the following schematic representation in terms of partial one-row diagrams:
\begin{figure}[H]
    \begin{equation*}
\begin{array}{c}
\begin{tikzpicture}
\node at (0.55,0) {$\dots$};
\draw (1.2,0) .. controls (1.7,0.7) and (3.1,0.7) .. (3.6,0);
\draw (1.6,0) .. controls (1.9,0.55) and (2.9,0.55) .. (3.2,0);
\draw (2.0,0) .. controls (2.1,0.4) and (2.7,0.4) .. (2.8,0);
\draw [dotted] (2.4,0.7) --(2.4,-1.7);
\node at (4.35,0.0) {$\dots$};
\node at (1.6,0) {\tiny\textbullet};
\node at (2.0,0) {\tiny\textbullet};
\node at (1.2,0) {\tiny\textbullet};
\node at (3.6,0) {\tiny\textbullet};
\node at (3.2,0) {\tiny\textbullet};
\node at (2.8,0) {\tiny\textbullet};
\draw (1.2,0) --(1.6,-0.4);
\draw (1.6,0) --(1.2,-0.4);
\draw (2.0,0) --(2.0,-0.4);
\draw (3.2,0) --(3.6,-0.4);
\draw (3.6,0) --(3.2,-0.4);
\draw (2.8,0) --(2.8,-0.4);

\draw (2.0,-0.4) --(1.6,-0.8);
\draw (1.6,-0.4) --(2.0,-0.8);
\draw (1.2,-0.4) --(1.2,-1.55);
\draw (3.2,-0.4) --(2.8,-0.8);
\draw (2.8,-0.4) --(3.2,-0.8);
\draw (3.6,-0.4) --(3.6,-1.55);

\draw (2.0,-0.8) .. controls (2.1,-1.2) and (2.7,-1.2) .. (2.8,-0.8);
\draw (2.0,-1.55) .. controls (2.1,-1.15) and (2.7,-1.15) .. (2.8,-1.55);
\draw (3.2,-0.8) --(3.2,-1.55);
\draw (1.6,-0.8) --(1.6,-1.55);

\node at (1.6,-1.55) {\tiny\textbullet};
\node at (2.0,-1.55) {\tiny\textbullet};
\node at (1.2,-1.55) {\tiny\textbullet};
\node at (3.6,-1.55) {\tiny\textbullet};
\node at (3.2,-1.55) {\tiny\textbullet};
\node at (2.8,-1.55) {\tiny\textbullet};
\node at (5.3,-0.6) {$-$};
\node at (6.4,-0.6) {$\delta$};
\node at (6.65,-0.6) {$\cdot$};
\node at (7.15,-0.8) {$\dots$};
\node at (11.6,-0.6) {$=$};
\node at (12.1,-0.6) {$0$};
\node at (12.45,-0.75) {$,$};
\draw (7.8,-0.8) .. controls (8.3,-0.1) and (9.7,-0.1) .. (10.2,-0.8);
\draw (8.2,-0.8) .. controls (8.5,-0.25) and (9.5,-.25) .. (9.8,-0.8);
\draw (8.6,-0.8) .. controls (8.7,-0.4) and (9.3,-0.4) .. (9.4,-0.8);
\draw [dotted] (9.0,-0.9) --(9.0,-0.1);
\node at (10.95,-0.8) {$\dots$};
\node at (8.2,-0.8) {\tiny\textbullet};
\node at (8.6,-0.8) {\tiny\textbullet};
\node at (7.8,-0.8) {\tiny\textbullet};
\node at (10.2,-0.8) {\tiny\textbullet};
\node at (9.8,-0.8) {\tiny\textbullet};
\node at (9.4,-0.8) {\tiny\textbullet};

\end{tikzpicture}
\end{array}
\end{equation*}
\end{figure}
\vskip -0.5cm
\begin{figure}[H]
    \begin{equation*}
\begin{array}{c}
\begin{tikzpicture}
\node at (0.55,0) {$\dots$};
\draw (1.2,0) .. controls (1.7,0.7) and (3.1,0.7) .. (3.6,0);
\draw (1.6,0) .. controls (1.9,0.55) and (2.9,0.55) .. (3.2,0);
\draw (2.0,0) .. controls (2.1,0.4) and (2.7,0.4) .. (2.8,0);
\draw [dotted] (2.4,0.7) --(2.4,-1.7);
\node at (4.35,0.0) {$\dots$};
\node at (1.6,0) {\tiny\textbullet};
\node at (2.0,0) {\tiny\textbullet};
\node at (1.2,0) {\tiny\textbullet};
\node at (3.6,0) {\tiny\textbullet};
\node at (3.2,0) {\tiny\textbullet};
\node at (2.8,0) {\tiny\textbullet};
\draw (1.2,0) --(1.2,-0.4);
\draw (1.6,0) --(1.6,-0.4);
\draw (2.0,0) --(2.0,-0.4);
\draw (3.2,0) --(3.6,-0.4);
\draw (3.6,0) --(3.2,-0.4);
\draw (2.8,0) --(2.8,-0.4);

\draw (2.0,-0.4) --(1.6,-0.8);
\draw (1.6,-0.4) --(2.0,-0.8);
\draw (1.2,-0.4) --(1.2,-1.55);
\draw (3.2,-0.4) --(2.8,-0.8);
\draw (2.8,-0.4) --(3.2,-0.8);
\draw (3.6,-0.4) --(3.6,-1.55);

\draw (2.0,-0.8) .. controls (2.1,-1.2) and (2.7,-1.2) .. (2.8,-0.8);
\draw (2.0,-1.55) .. controls (2.1,-1.15) and (2.7,-1.15) .. (2.8,-1.55);
\draw (3.2,-0.8) --(3.2,-1.55);
\draw (1.6,-0.8) --(1.6,-1.55);

\node at (1.6,-1.55) {\tiny\textbullet};
\node at (2.0,-1.55) {\tiny\textbullet};
\node at (1.2,-1.55) {\tiny\textbullet};
\node at (3.6,-1.55) {\tiny\textbullet};
\node at (3.2,-1.55) {\tiny\textbullet};
\node at (2.8,-1.55) {\tiny\textbullet};
\node at (5.3,-0.6) {$-$};
\node at (6.6,-0.8) {$\dots$};
\node at (11.0,-0.6) {$=$};
\node at (11.5,-0.6) {$0$};
\node at (11.85,-0.75) {$,$};
\draw (7.15,-0.8) .. controls (7.65,-0.1) and (9.05,-0.1) .. (9.55,-0.8);
\draw (7.55,-0.8) .. controls (7.85,-0.25) and (8.85,-.25) .. (9.15,-0.8);
\draw (7.95,-0.8) .. controls (8.05,-0.4) and (8.65,-0.4) .. (8.75,-0.8);
\draw [dotted] (8.35,-0.9) --(8.35,-0.1);
\node at (10.25,-0.8) {$\dots$};
\node at (7.55,-0.8) {\tiny\textbullet};
\node at (7.95,-0.8) {\tiny\textbullet};
\node at (7.15,-0.8) {\tiny\textbullet};
\node at (9.55,-0.8) {\tiny\textbullet};
\node at (9.15,-0.8) {\tiny\textbullet};
\node at (8.75,-0.8) {\tiny\textbullet};

\end{tikzpicture}
\end{array}
\end{equation*}
\end{figure}
\vskip -0.5cm
\begin{figure}[H]
    \begin{equation*}
\begin{array}{c}
\begin{tikzpicture}
\node at (0.55,0) {$\dots$};
\draw (1.2,0) .. controls (1.7,0.7) and (3.1,0.7) .. (3.6,0);
\draw (1.6,0) .. controls (1.9,0.55) and (2.9,0.55) .. (3.2,0);
\draw (2.0,0) .. controls (2.1,0.4) and (2.7,0.4) .. (2.8,0);
\draw [dotted] (2.4,0.7) --(2.4,-1.7);
\node at (4.35,0.0) {$\dots$};
\node at (1.6,0) {\tiny\textbullet};
\node at (2.0,0) {\tiny\textbullet};
\node at (1.2,0) {\tiny\textbullet};
\node at (3.6,0) {\tiny\textbullet};
\node at (3.2,0) {\tiny\textbullet};
\node at (2.8,0) {\tiny\textbullet};
\draw (1.2,0) --(1.6,-0.4);
\draw (1.6,0) --(1.2,-0.4);
\draw (2.0,0) --(2.0,-0.4);
\draw (3.2,0) --(3.2,-0.4);
\draw (3.6,0) --(3.6,-0.4);
\draw (2.8,0) --(2.8,-0.4);

\draw (2.0,-0.4) --(1.6,-0.8);
\draw (1.6,-0.4) --(2.0,-0.8);
\draw (1.2,-0.4) --(1.2,-1.55);
\draw (3.2,-0.4) --(2.8,-0.8);
\draw (2.8,-0.4) --(3.2,-0.8);
\draw (3.6,-0.4) --(3.6,-1.55);

\draw (2.0,-0.8) .. controls (2.1,-1.2) and (2.7,-1.2) .. (2.8,-0.8);
\draw (2.0,-1.55) .. controls (2.1,-1.15) and (2.7,-1.15) .. (2.8,-1.55);
\draw (3.2,-0.8) --(3.2,-1.55);
\draw (1.6,-0.8) --(1.6,-1.55);

\node at (1.6,-1.55) {\tiny\textbullet};
\node at (2.0,-1.55) {\tiny\textbullet};
\node at (1.2,-1.55) {\tiny\textbullet};
\node at (3.6,-1.55) {\tiny\textbullet};
\node at (3.2,-1.55) {\tiny\textbullet};
\node at (2.8,-1.55) {\tiny\textbullet};
\node at (5.3,-0.6) {$-$};
\node at (6.6,-0.8) {$\dots$};
\node at (11.0,-0.6) {$=$};
\node at (11.5,-0.6) {$0$};
\node at (11.85,-0.75) {$,$};
\draw (7.15,-0.8) .. controls (7.65,-0.1) and (9.05,-0.1) .. (9.55,-0.8);
\draw (7.55,-0.8) .. controls (7.85,-0.25) and (8.85,-.25) .. (9.15,-0.8);
\draw (7.95,-0.8) .. controls (8.05,-0.4) and (8.65,-0.4) .. (8.75,-0.8);
\draw [dotted] (8.35,-0.9) --(8.35,-0.1);
\node at (10.25,-0.8) {$\dots$};
\node at (7.55,-0.8) {\tiny\textbullet};
\node at (7.95,-0.8) {\tiny\textbullet};
\node at (7.15,-0.8) {\tiny\textbullet};
\node at (9.55,-0.8) {\tiny\textbullet};
\node at (9.15,-0.8) {\tiny\textbullet};
\node at (8.75,-0.8) {\tiny\textbullet};

\end{tikzpicture}
\end{array}
\end{equation*}
\end{figure}
\vskip -0.5cm
\begin{figure}[H]
    \begin{equation*}
\begin{array}{c}
\begin{tikzpicture}
\node at (0.2,0) {$\dots$};
\draw (1.2,0) .. controls (1.7,0.7) and (3.1,0.7) .. (3.6,0);
\draw (1.6,0) .. controls (1.9,0.55) and (2.9,0.55) .. (3.2,0);
\draw (2.0,0) .. controls (2.1,0.4) and (2.7,0.4) .. (2.8,0);
\draw [dotted] (2.4,0.7) --(2.4,-1.7);
\node at (4.7,0.0) {$\dots$};
\node at (1.6,0) {\tiny\textbullet};
\node at (2.0,0) {\tiny\textbullet};
\node at (1.2,0) {\tiny\textbullet};
\node at (3.6,0) {\tiny\textbullet};
\node at (3.2,0) {\tiny\textbullet};
\node at (2.8,0) {\tiny\textbullet};
\node at (4,0) {\tiny\textbullet};
\node at (0.8,0) {\tiny\textbullet};
\draw (1.2,0) --(0.8,-0.4);
\draw (0.8,0) --(1.2,-0.4);
\draw (0.8,0) --(0.8,0.7);
\draw (0.8,-0.4) --(0.8,-1.95);
\draw (1.6,0) --(1.6,-0.4);
\draw (2.0,0) --(2.0,-0.4);

\draw (3.6,0) --(4.0,-0.4);
\draw (4.0,0) --(3.6,-0.4);
\draw (4.0,0) --(4.0,0.7);
\draw (4.0,-0.4) --(4.0,-1.95);
\draw (2.8,0) --(2.8,-0.4);
\draw (3.2,0) --(3.2,-0.4);

\draw (1.2,-0.4) --(1.6,-0.8);
\draw (1.6,-0.4) --(1.2,-0.8);
\draw (2.0,-0.4) --(2.0,-0.8);
\draw (3.2,-0.4) --(3.6,-0.8);
\draw (3.6,-0.4) --(3.2,-0.8);
\draw (2.8,-0.4) --(2.8,-0.8);

\draw (2.0,-0.8) --(1.6,-1.2);
\draw (1.6,-0.8) --(2.0,-1.2);
\draw (1.2,-0.8) --(1.2,-1.95);
\draw (3.2,-0.8) --(2.8,-1.2);
\draw (2.8,-0.8) --(3.2,-1.2);
\draw (3.6,-0.8) --(3.6,-1.95);

\draw (2.0,-1.2) .. controls (2.1,-1.6) and (2.7,-1.6) .. (2.8,-1.2);
\draw (2.0,-1.95) .. controls (2.1,-1.55) and (2.7,-1.55) .. (2.8,-1.95);
\draw (3.2,-1.2) --(3.2,-1.95);
\draw (1.6,-1.2) --(1.6,-1.95);

\node at (5.5,-0.7) {$=$};
\node at (6.0,-0.7) {$0$};
\node at (6.35,-0.85) {$.$};

\node at (1.6,-1.95) {\tiny\textbullet};
\node at (2.0,-1.95) {\tiny\textbullet};
\node at (1.2,-1.95) {\tiny\textbullet};
\node at (3.6,-1.95) {\tiny\textbullet};
\node at (3.2,-1.95) {\tiny\textbullet};
\node at (2.8,-1.95) {\tiny\textbullet};
\node at (0.8,-1.95) {\tiny\textbullet};
\node at (4.0,-1.95) {\tiny\textbullet};

\end{tikzpicture}
\end{array}
\end{equation*}
\end{figure}
\vskip -0.5cm
\noindent 

Let $\Upsilon^k_f\subset\mathsf{Sh}^R_{f,s-f}\Sn{f}{R}\Sigma^R_f$ be a subset of all monomials in $\Sn{s}{R}$ which include only generators $s_{r+k},\dots,s_{r+s-1}$. For brevity denote by $\wpp{r+1}^{(i)}$ for $i = 1\ldots f$ the set of words
\begin{equation}
  [r+1+k_1,r+1]\dots[r+i+k_i,r+i],\,\,\,0\leqslant k_1<s-1,\dots , 0\leqslant k_i<s-i . 
\end{equation}
We set $\wpp{r+1}^{(0)} = \left\{1\right\}$. 

\vskip .1cm
We construct the following sets ($t=1\ldots\min(r,s)$)  of elements of the algebra $\wB{r}{s}$: 
\begin{align}
    &\label{eq:loops1} \bar{\mathfrak{B}}_{r,s}^{(t)}\ast\bigcup\nolimits_{i=0}^{f-1}\left([r,r-i]\wpp{r+1}^{(i)} -\delta\right)\Upsilon_{f}^{i+2},\\
    &\label{eq:loops2}  \bar{\mathfrak{B}}_{r,s}^{(t)}\ast\bigcup\nolimits_{i=0}^{\min(f,s-1)-1}\left([r,r-i]\wpp{r+1}^{(i+1)}-1\right)\Upsilon_{f}^{i+2},\\
    &\label{eq:loops3} \bar{\mathfrak{B}}_{r,s}^{(t)}\ast\bigcup\nolimits_{i = 1}^{\min(f,r-1)}\bigcup\nolimits_{j=i+1}^r\left([r,j-i]\wpp{r+1}^{(i-1)}-1\right)\Upsilon_{f}^{i+1},
        \\
    &\label{eq:loops4} \bar{\mathfrak{B}}_{r,s}^{(t)}\ast[r,r-f]\wpp{r+1}^{(f)}\Upsilon_{f}^{f+1}.  
\end{align}
\vskip 0.2cm
\noindent
A direct inspection shows that the elements \eqref{eq:loops1}-\eqref{eq:loops4} annihilate the vector $v_f$ since the elements \eqref{eq:loop1}-\eqref{eq:loop4} do.

\paragraph{Part 2} Consider the set $\{s_{r+i}-s_{r-i},\,i=1,\dots, f-1\}$. The elements of this set annihilate the vector $v_f$, see figure below
\begin{figure}[H]
    \begin{equation*}
\begin{array}{c}
\begin{tikzpicture}
\node at (0.55,0) {$\dots$};
\draw (1.2,0) .. controls (1.7,0.7) and (3.1,0.7) .. (3.6,0);
\draw (1.6,0) .. controls (1.9,0.55) and (2.9,0.55) .. (3.2,0);
\node at (2.4,0) {$\dots$};
\draw [dotted] (2.4,0.7) --(2.4,-0.5);
\node at (4.35,0.0) {$\dots$};
\node at (1.6,0) {\tiny\textbullet};

\node at (1.2,0) {\tiny\textbullet};
\node at (3.6,0) {\tiny\textbullet};
\node at (3.2,0) {\tiny\textbullet};

\draw (1.6,0) --(1.6,-0.4);
\draw (1.2,0) --(1.2,-0.4);
\draw (3.2,0) --(3.6,-0.4);
\draw (3.6,0) --(3.2,-0.4);

\node at (1.6,-0.4) {\tiny\textbullet};
\node at (1.2,-0.4) {\tiny\textbullet};
\node at (3.6,-0.4) {\tiny\textbullet};
\node at (3.2,-0.4) {\tiny\textbullet};

\node at (6.3,0) {$\dots$};
\draw (6.85,0) .. controls (7.35,0.7) and (8.75,0.7) .. (9.25,0);
\draw (7.25,0) .. controls (7.55,0.55) and (8.55,0.55) .. (8.85,0);
\node at (9.95,0) {$\dots$};
\draw [dotted] (8.05,0.7) --(8.05,-0.5);
\node at (8.11,0.0) {$\dots$};
\node at (6.85,0) {\tiny\textbullet};
\node at (9.25,0) {\tiny\textbullet};
\node at (7.25,0) {\tiny\textbullet};
\node at (8.85,0) {\tiny\textbullet};

\node at (5.3,-0.1) {$-$};
\node at (10.7,-0.1) {$=$};
\node at (11.2,-0.1) {$0$};
\node at (11.55,-0.25) {$.$};

\draw (6.85,0) --(7.25,-0.4);
\draw (7.25,0) --(6.85,-0.4);
\draw (8.85,0) --(8.85,-0.4);
\draw (9.25,0) --(9.25,-0.4);

\node at (6.85,-0.4) {\tiny\textbullet};
\node at (9.25,-0.4) {\tiny\textbullet};
\node at (7.25,-0.4) {\tiny\textbullet};
\node at (8.85,-0.4) {\tiny\textbullet};

\end{tikzpicture}
\end{array}
\end{equation*}
\end{figure}
\vskip -0.5cm 
\noindent
Given a word ${\sf x}$ in $\Sn{f}{R}\backslash \{1\}$ let $s_{r+i}$ be its leftmost generator ($i=1\ldots f-1$). Denote by ${\sf x}_c$ the element of the algebra $\wB{r}{s}$
obtained by replacing the letter $s_{r+i}$ in the word ${\sf x}$ by the combination $(s_{r+i}-s_{r-i})$. Define the set $\overline{\mathfrak{S}}^{R}_{f}$ constituted by 
elements ${\sf x}_c$, ${\sf x}\in \Sn{f}{R}\backslash \{1\}$. The elements of the set
\begin{equation}\label{eq:arcs}
        \Theta_f\overline{\mathfrak{S}}^{R}_{f}\Sigma_f
\end{equation}
annihilate the vector $v_f$ as well.

\paragraph{Part 3} We recall some results from \cite{P}. Let $S(\lambda)$ be the Specht module for the symmetric group $\Sn{n}{}$ for some $n$. 
Consider the vector in $S(\lambda)$ corresponding to the tableau $\check{t}$ filled with numbers $1\dots n$ in natural order reading down the column from left to right. The annihilator ideal of $\check{t}$ is the left ideal generated by the Garnir elements and $1+\tau$ where $\tau$ are transpositions in the column stabiliser of the tableau.

\vskip .1cm
Denote $\mathfrak{g}^L_f$ (respectively, $\mathfrak{g}^R_f$) a basis of the annihilator ideal of the vector $\check{t}_f^{L}$ in $S\left(\lambda^L_f\right)$ (respectively, $\check{t}_f^{R}$ in $S\left(\lambda^R_f\right)$). The following elements of the algebra $\wB{r}{s}$
\begin{align}
        &\label{eq:loops_g} \Sn{r}{L}\mathfrak{D}_{r,s}^{(f)}\mathsf{Sh}^R_{f,s-f}\,\Sn{f}{R}\mathfrak{g}^R_f,\\
        &\label{eq:LR} \Theta_f\,\Sn{f}{R}\,\left(\mathfrak{g}^L_f\Sigma^R_f\cup \Sigma^L_f\mathfrak{g}^R_f\cup\mathfrak{g}^L_f\mathfrak{g}^R_f\right)
\end{align}
\vskip 0.1cm
\noindent
annihilate the vector $v_f$ because they annihilate $\check{t}_f^{L}$ and $\check{t}_f^{R}$.

\vskip .1cm
Let $A_f$ be the union of all sets \eqref{eq:loops1}-\eqref{eq:loops4}, \eqref{eq:arcs}, \eqref{eq:loops_g}, \eqref{eq:LR}. The following Lemma holds.
\begin{lemma}\label{annihilator}
The set $A_f$ is a basis of annihilator ideal of the vector $v_f$,
\begin{equation}
    \# A_f = \dim\wB{r}{s} - \dim C_{r,s}(\plambda_f).
\end{equation}
\end{lemma}
\begin{proof}
First let us show that the sets \eqref{eq:loops1}-\eqref{eq:loops4} and \eqref{eq:loops_g} are linearly independent. For that purpose consider the `higher' terms in \eqref{eq:loops1}-\eqref{eq:loops3}: 

\begin{align}
    &\label{nn1} \bar{\mathfrak{B}}_{r,s}^{(t)}\ast\bigcup\nolimits_{i=0}^{f-1}\left([r,r-i]\wpp{r+1}^{(i)} \right)\Upsilon_{f}^{i+2},\\
    &\label{nn2}  \bar{\mathfrak{B}}_{r,s}^{(t)}\ast\bigcup\nolimits_{i=0}^{\min(f,s-1)-1}\left([r,r-i]\wpp{r+1}^{(i+1)}\right)\Upsilon_{f}^{i+2},\\
    &\label{nn3} \bar{\mathfrak{B}}_{r,s}^{(t)}\ast\bigcup\nolimits_{i = 1}^{\min(f,r-1)}\bigcup\nolimits_{j=i+1}^r\left([r,j-i]\wpp{r+1}^{(i-1)}\right)\Upsilon_{f}^{i+1}.
\end{align}
Let
$$M_i^{(1)}:=\bar{\mathfrak{B}}_{r,s}^{(t)}\ast\left([r,r-i]\wpp{r+1}^{(i)} \right)\Upsilon_{f}^{i+2}\ ,$$
so that the set in (\ref{nn1}) is the union of sets $M_i^{(1)}$, $i=0,1,\dots, f-1$. Similarly, let
$$M_i^{(2)}:=\bar{\mathfrak{B}}_{r,s}^{(t)}\ast\left([r,r-i]\wpp{r+1}^{(i+1)}\right)\Upsilon_{f}^{i+2}, $$
and
$$M_i^{(3)}:= \bar{\mathfrak{B}}_{r,s}^{(t)}\ast\bigcup\nolimits_{j=i+1}^r\left([r,j-i]\wpp{r+1}^{(i-1)}\right)\Upsilon_{f}^{i+1}.$$
The union of sets $M_i^{(1)}$, $M_i^{(2)}$ and $M_i^{(3)}$ for fixed $i$ ($0\leqslant i<f$) is
\vskip -0.5cm
\begin{multline}\label{eq:union}
\bar{\mathfrak{B}}_{r,s}^{(t)}\ast [r,r-i]\left(\wpp{r+1}^{(i)}\Upsilon_f^{i+2}\cup\wpp{r+1}^{(i+1)}\Upsilon_f^{i+2}\cup\left(\bigcup_{k=1}^i\wpp{r+1}^{(k-1)}\Upsilon_f^{k+1}\right)\right)\\
=\bar{\mathfrak{B}}_{r,s}^{(t)}\ast \w{r}{r-i}\Upsilon^{1}_{f},\quad t=1\ldots \min(r,s).
\end{multline}
\vskip 0.2cm
\noindent
For $i=f$ there are no sets \eqref{nn1} and \eqref{nn2}; the union of $M_f^{(1)}$ and the set (\ref{eq:loops4}) 
is, similarly to
\eqref{eq:union}: 
\begin{equation}\label{nbv1}\bar{\mathfrak{B}}_{r,s}^{(t)} \ast\w{r}{r-f}\Upsilon^{1}_{f}, \quad t=1\ldots \min(r,s).\end{equation} 
It follows from the definition of $\bar{\mathfrak{B}}_{r,s}^{(t)}$ that the union of the expressions \eqref{eq:union} for $i=0\ldots f-1$ and (\ref{nbv1}) forms a subset $\Sn{r}{L}\mathfrak{D}_{r,s}^{(t)}\Upsilon^{1}_{f}\subset \mathfrak{B}_{r,s}^{(t)}$, thus the expressions \eqref{eq:union} for $i=0\ldots f-1$ and (\ref{nbv1}) are linearly independent.
Moreover, we have, by construction, $\Upsilon^{1}_{f}=\mathsf{Sh}^R_{f,s-f}\,\Sn{f}{R}\Sigma^R_f$.  

\vskip .1cm 
The union of the expressions \eqref{eq:loops_g}, \eqref{eq:union} and (\ref{nbv1}) over $t=1\ldots \min(r,s)$ and $i = 0\ldots f$ will be denoted by $\mathcal{B}$. 
Since, by \cite{P}, the elements of $\mathfrak{g}^{L}_f$ (respectively, $\mathfrak{g}^{R}_f$) and $\Sigma^{L}_f$ (respectively, $\Sigma^{R}_f$) together form a basis in $\Sn{r}{}$ (respectively, $\Sn{s}{}$), we conclude that 
$\mathcal{B}$ is a linearly independent set.  The same is true for the expressions \eqref{eq:loops1}-\eqref{eq:loops4} and \eqref{eq:loops_g}. Indeed, each element in \eqref{eq:loops1}-\eqref{eq:loops3} is a combination of two monomials: the first one is a minimal word and contains more letters $s_r$ than the second. 
The number of occurrences of the letter $s_r$ in a word defines a filtration on the algebra $\wB{r}{s}$.  
Assume that the expressions \eqref{eq:loops1}-\eqref{eq:loops4} and \eqref{eq:loops_g} are not linearly independent. Choose then a shortest non-trivial linear dependency. 
The coefficients of the words containing the maximal number of letters $s_r$ are zero (because these are the ones from $\mathcal{B}$) contradicting to the 
minimality of length of of dependency.

\vskip .1cm
Each expression \eqref{eq:arcs} is a sum of two words from the set $\Theta_f\Sn{f}{R}\Sigma_f$, one containing more generators from $\Sn{r}{R}$ than the other. 
This implies the linear independence of the set  \eqref{eq:arcs}.

\vskip .1cm
Next, we move to showing that the union of the sets \eqref{eq:arcs} and \eqref{eq:LR} is linearly independent. First, note that 
replacing in the expressions \eqref{eq:arcs} the elements from $\overline{\mathfrak{S}}_f^R$ by their pullbacks from $\Sn{f}{R}\backslash \{1\}$ we obtain a set $\mathcal{N}$ whose union with the expressions \eqref{eq:LR}  
is disjoint and equals $\left(\Theta_f\Sn{f}{R}\Sigma_f\right)\backslash X_f$. The set $\Theta_f\Sn{f}{R}\Sigma_f$
is a basis in $\mathbb{C}\left[\Sn{r}{}\times\Sn{s}{}\right]$. Therefore the union of the sets $\mathcal{N}$ and  \eqref{eq:LR} is linearly independent.  
The argument appealing to the length, defined by the number of generators from $\Sn{r}{R}$, completes the proof of the linear independence of union of the sets \eqref{eq:arcs} and \eqref{eq:LR}.

\vskip .1cm
The expressions from the sets \eqref{eq:arcs} and \eqref{eq:LR} do not contain generators $s_r$ and therefore the whole set $A_f$ consists of linearly independent elements.
 
 \vskip .1cm
To calculate $\# A_f$, note that: 
\begin{itemize}
\item[a)] $\# \mathcal{B} = (r+s)! - r!s!$ because $\mathcal{B}=\mathfrak{B}_{r,s}\backslash \mathfrak{B}_{r,s}^{(0)}$, 
\item[b)] the cardinality of the union of the sets \eqref{eq:arcs} and \eqref{eq:LR} is $r!s!-\dim C_{r,s}(\plambda_f)$ (see Lemma \ref{genset}).
\end{itemize}
As a result, we arrive at the correct cardinality for the annihilator ideal $\# A = \dim\wB{r}{s} - \dim C_{r,s}(\plambda_f)$.
\end{proof}
\vskip 0.2cm
\noindent
Lemmas \ref{genset} and \ref{annihilator} together provide a constructive proof of the Theorem \ref{decomp}.

\section*{Acknowledgements}
D.B. is grateful to A.~Kiselev for participation at the early stage of the project. The work of D.B. is funded by Excellence Initiative of Aix-Marseille University -- A*MIDEX and Excellence Laboratory Archimedes LabEx, French ``Investissements d'Avenir'' programmes. The work of Y.G. is supported by a joint grant ``50/50'' UMONS -- Universit\'e Fran\c{c}ois Rabelais de Tours. The work of O. O. was supported by the Program of Competitive Growth of Kazan Federal University and by the ANR grant ``Phymath".

\section*{Appendix}\label{Append}
\subsection*{A: partial order on $\wBfree{r}{s}$}\label{Order}
We introduce a partial order $\bec$ on the set of monomials from the monoid $\wBfree{r}{s}$ which is compatible with $\mathfrak{R}^\prime$. For a generator $\hat{s}_i$, $i=1,\ldots ,r+s-1$, and a monomial $\sfw\in\wBfree{r}{s}$ we shall write $\hat{s}_i\in\sfw$ (respectively, $\hat{s}_i\notin\sfw$) whenever $\hat{s}_i$ is (respectively, is not) contained in $\sfw$; $\pi_{L}(\sfw)$ (respectively, $\pi_{R}(\sfw)$)  means the substitution $\hat{s}_i\to 1$ for all generators with $i>r$ (respectively, $i<r$); the unit element $1$ stands for the empty word. For a word $\sfw\in\wBfree{r}{s}$, $|\sfw|$ means the length of $\sfw$, while $|\sfw|_i$ denotes the number of occurrences of $\hat{s}_i$ in $\sfw$. The symbol $|\sfw|_{L}$ (respectively, $|\sfw|_{R}$) denotes the number of generators $\hat{s}_i\in\sfw$ with $i=1,\ldots ,r-1$ (respectively, $r+1,\ldots ,r+s-1$). By definition, $|1| = |1|_i = |1|_L = |1|_R = 0$. 

\vskip .1cm
First we introduce the following partial order $\bbec$ on the subset $\Snfree{r}{s}\subset \wBfree{r}{s}$ constituted by all monomials such that $\sfw = \pi_L(\sfw)\pi_R(\sfw)$. First, compare the lengths of monomials $\sfu,\sfv\in\Snfree{r}{s}$: $|\sfu|_L<|\sfv|_L$ $\Rightarrow$ $\sfu\bbec\sfv$, while if $|\sfu|_L=|\sfv|_L$ then $|\sfu|_R<|\sfv|_R$ $\Rightarrow$ $\sfu\bbec\sfv$. In case $|\sfu| = |\sfv| = N>0$ and $|\sfu|_L = |\sfv|_L = N^\prime >0$ compare $\pi_L(\sfu) = \hat{s}_{i_1}\ldots \hat{s}_{i_{N^\prime}}$ and $\pi_L(\sfv) = \hat{s}_{j_1}\ldots \hat{s}_{j_{N^\prime}}$ lexicographically from left to right: let $k = 1\ldots N^\prime$ be the minimal number such that $i_p=j_p$ for all $p<k$ but not for $k$. Then $i_k<j_k$ implies $\sfu\bbec\sfv$. If $k$ doesn't exist, as well as if $N^\prime = 0$, compare $\pi_R(\sfu) = \hat{s}_{m_1}\ldots \hat{s}_{m_{N^{\prime\prime}}}$ and $\pi_R(\sfv) = \hat{s}_{n_1}\ldots \hat{s}_{n_{N^{\prime\prime}}}$ lexicographically from right to left: let $l=1\ldots N^{\prime\prime}$ be the maximal number such that $m_q = n_q$ for all $q>l$ but not for $l$. Then $m_l>n_l$ implies $\sfu\bbec\sfv$. Otherwise, $\sfu,\sfv\in\Snfree{r}{s}$ are incomparable with respect to $\bbec$.

To compare $\sfu,\sfv\in\wBfree{r}{s}$ one proceeds consecutively by the following steps. If monomials $\sfu,\sfv$ do not meet any conditions at a given step then one moves to the next step.
\begin{itemize}
    \item[{\bf i.}] Compare the lengths of monomials: $|\sfu|<|\sfv|$ $\Rightarrow$ $\sfu\bec\sfv$.
    \item[{\bf ii.}] If $|\sfu|=|\sfv|$ then compare $|\sfu|_R$ and $|\sfv|_R$: $|\sfu|_r < |\sfv|_r$ $\Rightarrow$ $\sfu\bec\sfv$.
    \item[{\bf iii.}] Let $|\sfu|=|\sfv|$, $|\sfu|_R=|\sfv|_R$. Compare $\sfu = \hat{s}_{i_1}\ldots \hat{s}_{i_N}$ and $\sfv = \hat{s}_{j_1}\ldots \hat{s}_{j_N}$ lexicographically from left to right: let $k=1\ldots N$ be the minimal number such that either $i_{k^\prime},j_{k^\prime}<r$ or $i_{k^\prime},j_{k^\prime}>r$ or $i_{k^\prime}=j_{k^\prime}=r$ for all $k^\prime<k$ and not for $k$. If either $i_k\leqslant r < j_k$ or $i_k< r \leqslant j_k$ then $\sfu\bec\sfv$.  
    \item[{\bf iv.}] Let $\sfu = \sfu_0 \hat{s}_r \sfu_1 \hat{s}_r\ldots \hat{s}_r \sfu_H$ and $\sfv = \sfv_0 \hat{s}_r \sfv_1 \hat{s}_r\ldots \hat{s}_r \sfv_H$ 
    and $\hat{s}_r\notin \sfu_\alpha,\sfv_\alpha$, $\alpha = 0,\ldots ,H$ ($H=0$ implies $\hat{s}_r\notin \sfu,\sfv$).  Compare $\pi_L(\sfu_{\alpha})\pi_R(\sfu_{\alpha}), \pi_L(\sfv_{\alpha})\pi_R(\sfv_{\alpha})\in\Snfree{r}{s}$ with respect to $\bbec$: let $\alpha_0 = 0\ldots H$ be the minimal number such that $\pi_L(\sfu_{\beta})\pi_R(\sfu_{\beta}), \pi_L(\sfv_{\beta})\pi_R(\sfv_{\beta})$ are incomparable with respect to $\bbec$ for all $\beta<\alpha_0$ but not for $\alpha_0$. Then the relation $\pi_L(\sfu_{\alpha_0})\pi_R(\sfu_{\alpha_0})\bbec \pi_L(\sfv_{\alpha_0})\pi_R(\sfv_{\alpha_0})$ implies $\sfu\bec \sfv$. 
    \item[{\bf v.}] In all other cases elements $\sfu,\sfv$ are incomparable with respect to $\bec$.
\end{itemize}

\subsection*{B: multiplication by generators}\label{Mul}
We write down the left multiplication of elements from $\mathfrak{B}_{r,s}$  by the generators $s_p\in\wB{r}{s}$. Let $[1,1-q_1]\dots[r-1,r-1-q_{r-1}]\in \Sn{r}{L}$ with $-1\leqslant q_1<1,\ldots,-1\leqslant q_{r-1}<r-1$. We will use a shorthand notation $[\dots][r-1,r-1-q]$ with $[\dots]=[1,1-q_1]\dots[r-2,r-2-q_{r-2}]$. For the products of the form $\w{r+i_1}{r-f+1 - j_{f}}\dots\w{r+f-1 + i_{f}}{r-j_{1}}$ we impose by default $s-f\geqslant i_{f}\geqslant \dots \geqslant i_{1} \geqslant 0$ and $r-f\geqslant j_{f}\geqslant \dots \geqslant j_1 \geqslant 0$, unless else is specified. For brevity we do not write out the well-known multiplication of the elements from $\Sn{r}{L}$ (respectively, $\Sn{s}{R}$) by the generator $s_i\in\Sn{r}{}$ (respectively, $s_i\in\Sn{s}{} $).
\paragraph{I.} $1\leqslant p<r$:
    \begin{equation*}
       s_p\cdot \Sn{r}{L}\mathfrak{D}_{r,s}^{(f)}\Sn{s}{R}=\left(s_p\cdot \Sn{r}{L}\right)\mathfrak{D}_{r,s}^{(f)}\Sn{s}{R} .
    \end{equation*}
\paragraph{II.}  $p=r$, $f>0$:
\paragraph{a)} $\Sn{r}{L}$ is represented by $[\ldots]$ and $i_1 = 0$:
    \begin{multline*}
       s_r\cdot [\dots]\w{r}{r-f+1 - j_{f}}\dots\w{r+f-1 + i_{f}}{r-j_{1}}\Sn{s}{R}\\
       =\delta [\dots]\w{r}{r-f+1 - j_{f}}\dots\w{r+f-1 + i_{f}}{r-j_{1}}\Sn{s}{R},
    \end{multline*}
\paragraph{b)} $i_1 = 0$ and $q\geqslant 0$:
    \begin{multline*}
       s_r\cdot [\dots] [r-1,r-1-q]\w{r}{r-f+1 - j_{f}}\dots\w{r+f-1 + i_{f}}{r-j_{1}}\Sn{s}{R}\\
       =\big([\dots]\cdot[r-2,r-1-q]\big)\w{r}{r-f+1 - j_{f}}\dots\w{r+f-1 + i_{f}}{r-j_{1}}\Sn{s}{R},
    \end{multline*}
\paragraph{c)} $s-f\geqslant i_{f}\geqslant \dots \geqslant i_{1} \geqslant 1$ and $r-f>q-f+1\geqslant j_{f}\geqslant \dots \geqslant j_1 \geqslant 0$:
\begin{multline*}
       s_r\cdot [\dots]\w{r-1}{r-1-q}\w{r+i_1}{r-f+1 - j_{f}}\dots\w{r+f-1 + i_{f}}{r-j_{1}}\Sn{s}{R}\\
       =[\dots]\w{r}{r-1 -q}\w{r+i_1}{r-f+1 - j_{f}}\dots\w{r+f-1 + i_{f}}{r-j_{1}}\Sn{s}{R},
    \end{multline*}
\paragraph{d)} $-1\leqslant q < j_f+f-1$ (put $i_0=1$) and $s-f\geqslant i_{f}\geqslant \dots \geqslant i_{1} \geqslant 1$:
\begin{multline*}
       s_r\cdot [\dots][r-1,r-1-q]\w{r+i_1}{r-f+1 - j_{f}}\dots\w{r+f-1 + i_{f}}{r-j_{1}}\Sn{s}{R}\\
       =[\dots]\w{r}{r-f+1 -j_f}\w{r+i_1}{r-f+2 - j_{f-1}}\dots\\
       \dots\w{r+q+i_{q+1}}{r-f+q+2-j_{f-q-1}}\w{r+q+2+i_{q+3}}{r-f+q+3-j_{f-q-2}}\dots\\
       \dots\w{r+f-1 + i_{f}}{r-j_{1}}\left(\w{r+f+i_{q+2}-1}{r+f+1}\cdot\Sn{s}{R}\right).
    \end{multline*}
\paragraph{III.} $p=r$, $f=0$ and $-1\leqslant q<r-1$:
\begin{equation*}
s_r\cdot [\dots][r-1,r-1-q]\Sn{s}{R}=[\dots]\w{r}{r-1-q}\Sn{s}{R}.
    \end{equation*}
\paragraph{IV.} $r<p\leqslant r+f+i_f$:
\paragraph{a)}  $p=r+k+i_k$, $i_{k+1}>i_{k}$:
\begin{multline*}
    s_p\cdot \Sn{r}{L}\dots\w{r+k-1+i_k}{r-f+k-j_{f-k+1}}\dots\Sn{s}{R}=\\
    \Sn{r}{L}\dots\w{r+k+i_k}{r-f+k-j_{f-k+1}}\dots\Sn{s}{R},
\end{multline*}
\paragraph{b)} $p=r+k+i_k$, $i_{k+1}=i_{k}$:
\begin{multline*}
    s_p\cdot \Sn{r}{L}\dots\w{r+k-1+i_k}{r-f+k-j_{f-k+1}}\dots\Sn{s}{R}=\\
    \left(\Sn{r}{L}\cdot s_{r-k}\right)\dots\w{r+k-1+i_k}{r-f+k-j_{f-k+1}}\dots\Sn{s}{R}.
\end{multline*}
\paragraph{V.} $p>r+f+i_f$:
\begin{equation*}
       s_p\cdot \Sn{r}{L}\mathfrak{D}_{r,s}^{(f)}\Sn{s}{R}=\Sn{r}{L}\mathfrak{D}_{r,s}^{(f)}\left(s_p\cdot \Sn{s}{R}\right).
    \end{equation*}


\begin{thebibliography}{99}
\bibitem[B]{B} G. M. Bergman, {\it The diamond lemma for ring theory}, Adv. Math. {\bf 29} (1978) 178--218.

\bibitem[BCHLLS]{BCHLLS} G.~Benkart, M.~Chakrabarti, T.~Halverson, R.~Leduc, C.~Lee and J.~Stroomer, {\it Tensor product
representations of general linear groups and their connections with Brauer algebras}, J. Algebra {\bf 166}:3 (1994) 529--567.

\bibitem[BE]{BE} C.~Bowman and J.~Enyang, {\it Skew cell modules for diagram algebras}; arXiv:1605.03448v2. 

\bibitem[BS]{BS} J. Brundan and C. Stroppel,
{\it Gradings on walled Brauer algebras and Khovanov's arc algebra}, Adv. Math. {\bf 231} (2012), 709-773.

\bibitem[BO]{BO} D. Bulgakova and O. Ogievetsky, {\it Fusion procedure for the walled Brauer algebra}, 
Journal of Geometry and Physics (2019).\\ DOI: https://doi.org/10.1016/j.geomphys.2019.103580 ;
arXiv:1911.10537 [math.RT] 


\bibitem[CDDM]{CVDM} A. Cox, M. De Visscher, S. Doty and P. Martin, {\it On the blocks of the walled Brauer algebra}, J. Algebra
{\bf 320}:1 (2008), 169--212.

\bibitem[FP]{FP} H. K. Farahat and M. H. Peel, {\it On the representation theory of Symmetric groups}, J. Algebra {\bf 67} (1980), 280-304.

\bibitem[H]{H} T. Halverson, {\it Characters of the centralizer algebras of mixed tensor representations of $\mathfrak{gl}(r,\mathbb{C})$ and the quantum group $U_q(\mathfrak{gl}(r,\mathbb{C}))$}, Pacific J. Math. {\bf 174}:2 (1996), 359--410.

\bibitem[JK]{JK} J. H. Jung and M. Kim, {\it Supersymmetric polynomials and the center of the walled Brauer algebra}; arXiv:1508.06469v2.

\bibitem[K]{K} K. Koike, {\it On the decomposition of tensor products of the representations of classical groups: by means of
the universal characters}, Adv. Math. {\bf 74}:1 (1989) 57--86.

\bibitem[KC]{KC} V. Kac, P. Cheung, {\it Quantum Calculus}, Springer-Verlag, New York (2002) 

\bibitem[KM]{KM} M. Kosuda and J. Murakami, {\it Centralizer algebras of the mixed tensor representations of quantum group $U_q(\mathfrak{gl}(n,\mathbb{C}))$}, Osaka J. Math. {\bf 30}:3 (1993), 475--507.

\bibitem[KT]{KT} C. Kassel and V. Turaev, {\it Braid Groups}, Springer GTM, volume 247, (2008).

\bibitem[N]{N} P. Nikitin, {\it Centralizer algebra of the diagonal action of the group $GL_n(\mathbb{C})$ in mixed tensor
space} [in Russian], POMI preprint 08/2006

\bibitem[NV]{NV} P. P. Nikitin and A. M. Vershik,  {\it Traces on Infinite-Dimensional Brauer Algebras}, Funct. Anal. Appl. {\bf 40}:3 (2006) 165--172.

\bibitem[OV]{OV} A. Okounkov and A. Vershik, {\it A new approach to representation theory of symmetric groups}, Selecta Math. {\bf 2}:4 (1996), 581-605.

\bibitem[P]{P} M. H. Peel, {\it Specht modules and the symmetric groups}, J. Algebra {\bf 36} (1975), 88--97.


\bibitem[T]{T} V. Turaev, {\it Operator invariants of tangles and R-matrices}, Math. USSR-Izv. {\bf 35}:2 (1990), 411--444.
\end{thebibliography}
\end{document}